%% file: uai2022.tex
\documentclass[accepted]{uai2022} 
\usepackage[american]{babel}

\usepackage{natbib} 
\usepackage{mathtools} 
\usepackage{booktabs} 
\usepackage{tikz} 

\usepackage{xcolor}
\usepackage{amsmath,amssymb,amsthm,mathrsfs,url,array}
\usepackage{booktabs} 
\usepackage{tikz} 
\usepackage{makecell}
\usepackage{diagbox}
\usepackage{multicol}
\usepackage{multirow}
\usepackage{graphicx}
\usepackage{subfigure}
\usepackage{appendix}
\usepackage[ruled,linesnumbered]{algorithm2e}
\newtheorem{Def}{Definition} 
\newtheorem{Th}{Theorem}
\newtheorem{Co}{Corollary}

\usepackage{scalerel,stackengine}
\stackMath
\newcommand\reallywidehat[1]{%
\savestack{\tmpbox}{\stretchto{%
  \scaleto{%
    \scalerel*[\widthof{\ensuremath{#1}}]{\kern-.6pt\bigwedge\kern-.6pt}%
    {\rule[-\textheight/2]{1ex}{\textheight}}
  }{\textheight}%
}{0.5ex}}%
\stackon[1pt]{#1}{\tmpbox}%
}

\title{Provable Constrained Stochastic Convex Optimization \\
with XOR-Projected Gradient Descent}

\author[1]{\href{mailto:<ding274@purdue.edu>?Subject=Your UAI 2022 paper}{Fan Ding}{}}
\author[2]{Yijie Wang}
\author[3]{Jianzhu Ma}
\author[1]{Yexiang Xue}
\affil[1]{%
    Computer Science Dept.\\
    Purdue University\\
    West Lafayette, Indiana, USA
}
\affil[2]{%
    Computer Science Dept.\\
    Indiana University\\
    Bloomington, Indiana, USA
  }
\affil[3]{%
    Institute for Artificial Intelligence \\
    Peking University \\
    Beijing, China
}

\begin{document}
\maketitle

\begin{abstract}
  \input{tex/abstract}
\end{abstract}

\section{INTRODUCTION}\label{sec:intro}
\input{tex/intro}

\section{PRELIMINARIES}\label{sec:prelim}
\input{tex/prelim}

\section{XOR-PROJECTED GRADIENT DESCENT}\label{sec:method}
\input{tex/method}

\section{EXPERIMENTS}\label{sec:exp}

\input{tex/exp}
\vspace{-0.3cm}
\section{CONCLUSION}\label{sec:conclusion}

\input{tex/conclusion}





\clearpage
\bibliography{fan}

\clearpage
\appendix
\input{tex/appendix}

\end{document}

%% file: tex/abstract.tex

Provably solving stochastic convex optimization problems with constraints is essential for various problems in science, business, and statistics.
Recently proposed XOR-Stochastic Gradient Descent (XOR-SGD) provides a  convergence rate guarantee solving the constraints-free version of the problem by leveraging XOR-Sampling.
However, the task becomes more difficult when additional equality and inequality constraints are needed to be satisfied. 
Here we propose XOR-PGD, a novel algorithm based on Projected Gradient Descent (PGD) coupled with the XOR sampler, which is guaranteed to solve the constrained stochastic convex optimization problem still in linear convergence rate by choosing proper step size.
We show on both synthetic stochastic inventory management and real-world road network design problems that the rate of constraints satisfaction of the solutions optimized by XOR-PGD is $10\%$ more than the competing approaches in very large searching space. The improved XOR-PGD algorithm is demonstrated to be more accurate and efficient than both XOR-SGD and SGD coupled with MCMC based samplers. It is also shown to be more scalable with respect to the number of samples and processor cores via experiments with large dimensions.

%% file: tex/intro.tex



Stochastic Convex Optimization problem is of great importance given its wide applicability in finance, control, robotics, management science, operational research, ecology, and conservation \citep{sodomka2007efficient,ziukov2016literature,gomes2019computational}. Advancements made to address this problem have ramifications in many domains. Mathematically, it minimizes an objective in expectation across multiple probabilistic scenarios under uncertainty.
\begin{align}
\begin{split}
    & \min_x~~\mathbb{E}_{\theta\sim Pr(\theta)}~~ f(x, \theta),  \\
    & s.t. \ \ x\in C:=\left \{\forall i,~ h_i(x)=0; ~ \forall j, ~ g_j(x)\leq 0\right \}.
    \label{eq:sto}
\end{split}
\end{align}
where each $f(x, \theta)$ is a a convex function with respect to $x$, $C$ is the set of constraints, $g_j(x)$ are convex functions and $h_i(x)$ are linear. 
Variable $\theta$ is sampled from distribution $Pr(\theta)$, which is represented as a Markov random field (MRF) in this paper.
We can see that this problem is highly intractable ($\#$P-hard) \citep{fan2021xorsgd}.
Despite the intractability of computing the expectation over a general probability distribution, a common operator in probabilistic inference, the problem is harder to solve because of the existence of the constraints set $C$. A recently proposed method, XOR-SGD \citep{fan2021xorsgd} first harness XOR-Sampler \citep{Ermon13Wish} to solve the constraints-free version of this task with guaranteed convergence rate using Stochastic Gradient Descent (SGD). However, they cannot provide such a bound once the solutions are in a constrained space which is more common in the real world. In this paper we consider the setting of the constrained stochastic convex optimization and first provide an algorithm, XOR-Projected Gradient Descent (XOR-PGD), that has a provably convergence rate towards the optimal value. 
The key of XOR-PGD is to first draw a representative set of samples from $Pr(\theta)$
which yield an accurate estimation of the gradient direction, and then adjust the estimation by projecting the variable $x$ to the constraints space $C$. Notice that even if PGD is a common method in traditional constrained optimization, it is not trivial to extend it to a stochastic setting where unbiased estimation is unavailable.

Like XOR-SGD, our XOR-PGD also leverages XOR-Sampling with a constant approximation guarantee,
which reduces the sampling problem into queries of NP oracles via 
hashing and projection.
Our key contribution is the extension of classical convergence analysis of 
XOR-SGD on constrained stochastic convex optimization problems, where we show that a constant multiplicative bound 
on the expectation of the gradient direction is sufficient to bound the 
final result against the true optimum.
Our theoretic contribution also does not depend on the unbiasness of the gradients, which was a necessary condition in previous analysis. 

On learning probabilistic graphical models, stochastic optimization is related to  \textit{the Marginal Maximum-a-posterior (MMAP) problem} \citep{Xue2016MarginalMAP,LiuI13VariationalMMAP,Radu14AndOrMMAP,Maua2012AnytimeMAP,Marinescu2015AOBB,domke2013learning}. These problems can be formulated as (albeit non-convex) stochastic optimization problems. 
\textit{Convergence analysis of gradient descent} has been studied for objectives with or without constraints \citep{wang2013variance,dubey2016variance,agarwal2017finding,lee2015distributed,ruder2016overview,jin2017accelerated,ge2015escaping,duchi2011adaptive,hinton2012neural,kingma2014adam,duchi2018minimax,allen2017katyusha,allen2018natasha}. The constrained setting is related to a convex-concave saddle point minimax optimization problem \citep{mokhtari2020convergence,xie2020lower,wang2020improved} where primal-dual methods are often used \citep{du2019linear,hamedani2018primal}. A fruitful line of Stochastic alternating direction method of multipliers (ADMM) \citep{ouyang2013stochastic,zheng2016fast,liu2017accelerated} has been proposed to deal with constrained optimization under uncertainty, yet this uncertainty is represented by a uniform distribution where the unbiased estimation is available.
Similar idea of stochastic convex optimization is also proposed by \cite{fan2021xorcd} for machine learning, where they leverage a XOR sampler to estimate the partition function in the learning process of an energy-based model.
However, none of them offer a theoretical guarantee of convergence rate in the setting of constrained stochastic convex optimization.

Experimental results reveal that XOR-PGD is effective in optimizing constrained convex stochastic functions. 
XOR-PGD outperforms competing 
solvers XOR-SGD and those running SGD with either MCMC, BP or BPChain samplers
on both the constrained stochastic inventory management 
and the constrained stochastic network design problems on real-world data under various conditions in not only accuracy and speed, but also the rate of constraints satisfaction.
In particular,  \textbf{90\% solutions obtained by XOR-PGD satisfy the constraints set $C$ even when the searching space is very large in the stochastic inventory management problem, approximately 10\% more that of competing methods. Besides, the improved XOR-PGD algorithm converges faster than XOR-SGD by accessing 20 less XOR samples in each iteration and is able to find better solutions in the stochastic network design problem.}
See the experiments section for more details.

\noindent\textbf{Notations}~~  For function $f:\mathbb{R}^d\rightarrow \mathbb{R}$, we call it $L$-smooth if 
for all $x,y$ in the convex domain $dom~f$,
 $   f(y) \leq f(x)+\nabla f(x)^T (y-x) +\frac{L}{2}||y-x||^2$. 
Denote $f^+(x)$ as the positive part of function $f(x)$. In other words, $f^+(x) = \max\{f(x), 0\}$. 
$f^-(x)$ is defined similarly. 
%
%
For a random vector $x$, we define $\mathbb{E}[x]$ as the element-wise expectation and the total variation $Var(x)=\mathbb{E}[||x||_2^2]-||\mathbb{E}[x]||_2^2$ where $||\cdot||_2^2$ is the square of $l_2$ norm.


%% file: tex/prelim.tex

\subsection{Probabilistic Distribution}
The probability distribution $Pr(\theta)$ can be defined in various different forms.
We consider $Pr(\theta)$ as a graphical model specified as a factor graph with $N=|V|$ discrete random variables $\theta_i\in \Theta_i, i \in V$ where $\Theta_i=\{0,1\}$. The global random vector $\theta=[\theta_1,\theta_2,\ldots,\theta_N]$ takes value in the cartesian product $\Theta = \Theta_1 \times \Theta_2 \times \cdots \times \Theta_N$. We consider a function over $\theta \in \Theta$ as follows:
\begin{align}\label{eq:mrf}
f(\theta)=\prod_{\alpha\in \mathcal{I}} \phi_{\alpha}(\{\theta\}_{\alpha})
\end{align}
which factors into potentials $\phi_{\alpha}:\{\theta\}_{\alpha}\rightarrow \mathbb{R}^{+}$, where $\mathcal{I}$ is the set of all the cliques of the graph, $\{\theta\}_{\alpha}\subseteq V$ is a subset of variables that the factor $\phi_{\alpha}$ depends on. We consider a normalized distribution $p(\theta) = \frac{1}{Z}f(\theta)$ where $Z$, the normalization constant, also known as the partition function, is defined as $Z=\sum_{\theta} f(\theta)$. The structure of $Pr(\theta)$ or the set $\mathcal{I}$ can be built from domain knowledge and potential functions can be learned from real-world data. The focus of this paper is not on how to construct the MRF but is how we solve problem in Equation \ref{eq:sto} in general when $Pr(\theta)$ is given in the form shown in Equation \ref{eq:mrf}. 

\subsection{XOR-SGD}
XOR-SGD outperforms SGD with MCMC based samplers in that it has a constant bound on the probability of drawn samples and further guarantee a constant bound between the expectation of the distribution formed by samples and the expectation of true distribution, without requiring exponentially large number of samples.
These samples are drawn by XOR-Sampling \citep{ermon2013embed}, a recently
proposed sampling scheme with a constant approximation guarantee, 
which reduces the sampling problem into queries of NP oracles via 
hashing and projection. Indeed, XOR-SGD requires accessing NP-oracle queries at each iteration.
Specifically, XOR-SGD is guaranteed to converge to a solution that is within a vanishing constant away from the true optimum in linear number of iterations, which is shown as Theorem \ref{Th:XOR-SGD}.

\begin{Th}\label{Th:XOR-SGD}\citep{fan2021xorsgd}
Let $\rho,\kappa$ be the constant approximation factor as in \cite{fan2021xorsgd}, function $f(x,\theta): \mathbb{R}^d\times\{0,1\}^n\rightarrow \mathbb{R}$ be a $L$-smooth convex function w.r.t. $x$. Denote $OPT=\min_x \mathbb{E}_{\theta\sim Pr(\theta)}f(x,\theta)$ as the global optimum. 
Let $\sigma^2=\max_x\{Var(\nabla_x f(x,\theta))\}$ and $\varepsilon^2=\max_x\{||\mathbb{E}[\nabla_x f(x, \theta)]||_2^2\}$. 
For any $1\leq\rho\kappa\leq\sqrt{2}$, step size $t\leq \frac{2-\rho^2\kappa^2}{L\rho\kappa}$ and sample size $N\geq1$, $\overline{x_K}$ is the output of XOR-SGD and $\mbox{obj} = \mathbb{E}_{\theta}[f(\overline{x_K}, \theta)]$ is the objective function value at $\overline{x_K}$. We have: 
\begin{align}\label{eq:main}
    \mathbb{E}_{\overline{x_K}}[\mbox{obj}] - OPT
     &\leq\frac{\rho\kappa||x_0-x^*||_2^2}{2tK}+\frac{t(\sigma^2+\varepsilon^2)}{N}.
\end{align}
\end{Th}

Theorem 1 states that in expectation, the difference in terms of the objective function values between the output of XOR-SGD algorithm $\overline{x_K}$ and the true optimum $OPT$ is bounded
by a term that scales inversely proportional to the number of SGD iterations $K$ and a tail term $\frac{t(\sigma^2+\epsilon^2)}{N}$. Although hard to compute, both $\sigma^2$ and $\epsilon^2$ are from the input and do not depend on the algorithm. However, although the authors provide a constrained version of XOR-SGD in their paper, they are not able to give a theoretical guarantee of convergence rate of this algorithm under the constrained setting, which is known to be more intractable due to the extra constraints on the solution space.

%% file: tex/method.tex
In this section we propose XOR-PGD, a new method to solve the stochastic convex optimization problems with extra constraints. 
XOR-PGD converges to solutions that are at most a constant  away from the true optimum in linear number of iterations.
The detailed procedure of XOR-PGD for is shown in Algorithm \ref{alg: XOR-PGD}.
To approximate the gradient $\nabla_x \mathbb{E}_\theta f(x_k, \theta)$ at step $k$, XOR-PGD draws $N$ samples $\theta_1, \dots, \theta_N$ from $Pr(\theta)$ using XOR-Sampling. 
Because XOR-Sampling has a failure rate, XOR-PGD repeatedly call XOR-Sampling until all $N$ samples are obtained successfully (line 4 -- 10). 
Once $\theta_1, \dots, \theta_N$ are obtained, XOR-PGD uses the empirical mean 
$\overline{g_k}=\frac{1}{N}\sum_{i=1}^N\nabla_x f(x_k,\theta_i)$
as an approximation for $\nabla_x \mathbb{E}_\theta f(x_k, \theta)$. 

Then, by setting the step size $t_k$ as $\frac{\rho\kappa}{\mu\kappa}$, we update $x_k$ by first a minus of $t_k\overline{g_k}$ and then a projection $\mathbb{P}_C(x)$  to project it to to constraints space $C$. Here we define $\mathbb{P}_C(x)$ in Equation \ref{eq:Pc}:
\begin{align}\label{eq:Pc}
    \mathbb{P}_C(x):=\arg \min_{y\in C}\frac{1}{2}\|x-y\|_2^2
\end{align}
Using this projection, we can prove that the output of XOR-PGD in expectation converges to the true optimum within a small constant distance at a linear speed w.r.t. the number of iterations $K$.

\begin{algorithm}[t]
   \caption{XOR-PGD}
   \label{alg: XOR-PGD}
   \LinesNumbered
   \KwIn{$f(x,\theta),\mu, w(\theta),K,N,l, b,\delta, P, \alpha, C$}
   Initialize $x_1$ for function $f(x,\theta)$\\
   \For{$k=1$ to $K-1$}{
        $i\gets 1$\\
       \While{$i\leq N$}{
            \small{$s\gets$ XOR-Sampling($w(\theta), l, b,\delta, P, \alpha$)}\\
            \If{$s \neq Failure$}  {
                $\theta_i\gets s$\\  
                $i\gets i+1$
            }
       }
       Compute $\overline{g_k}\gets\frac{1}{N}\sum_{i=1}^N\nabla_x f(x_k,\theta_i)$\\
       Compute $t_k \gets \frac{\rho\kappa}{\mu k}$\\
       Update $x_{k+1}\gets \arg \min_{y\in C}\frac{1}{2}\|x_{k}-t_k\overline{g_k}-y\|_2^2$\\
   }
   ${x_K}\gets \overline{x_K}=\frac{1}{K}\sum_{k=1}^K x_k$\\
   \textbf{return} ${x_K}$
\end{algorithm}

\begin{algorithm}[t]
   \caption{Improved XOR-PGD }
   \label{alg: improved XOR-PGD}
   \LinesNumbered
   ... as in Algorithm 1, except \\
   replacing line 12 by: $t_k \gets \frac{2\rho\kappa}{\mu (k+1)}$\\
   replacing line 15 by: ${x_K}\gets \reallywidehat{x_K}= \frac{2\sum_{k=1}^K k x_k}{K(K+1)} $
\end{algorithm}

\begin{Th}\label{Th:XOR_SGD_bound_constrained}
Let $f:\mathbb{R}^d\rightarrow \mathbb{R}$ be a $L$-smooth and $\mu$-strongly convex function and $x^*=\arg\min_{x\in C} f(x)$, where $C$ is a convex constraints set of $x$. In iteration $k$, $g_k$ is the estimated gradient, i.e., $x_{k+1}=\mathbb{P}_C(x_{k}-t_kg_k)$ where $Var(g_k)\leq \sigma^2$ and $\mathbb{P}_C(x):=\arg \min_{y\in C}\frac{1}{2}\|x-y\|_2^2$ is the projection of $x$ onto $C$. If $t_k = \frac{c}{\mu k}$ and there exists $1\leq c\leq\sqrt{2}$ s.t. $\frac{1}{c}[\nabla f(x_k)]^+ \leq \mathbb{E}[g_k^+]\leq c[\nabla f(x_k)]^+$ and $c[\nabla f(x_k)]^- \leq \mathbb{E}[g_k^-]\leq \frac{1}{c}[\nabla f(x_k)]^-$, then the convergence rate of the XOR-PGD is $O(\frac{\log{K}}{K})$:
\begin{align}
\mathbb{E}[f(\overline{x_K})] - f(x^*) \leq \frac{B}{2\mu K} (1+\log(K)),
\end{align}
where $\overline{x_K}=\frac{1}{K}\sum_{k=1}^{K} x_k$.
\end{Th}

\begin{proof} (Theorem \ref{Th:XOR_SGD_bound_constrained})
Because orthogonal projections contract distances, we have
\begin{align*}
    \|x_{k+1}-x^*\|^2 &\leq \|x_k - t_kg_k -x^*\|^2\\
    &= \|x_k-x^*\|^2 + t_k^2\|g_k\|^2 - 2t_k\langle g_k, x_k-x^* \rangle
\end{align*}
Taking expectation on both side, we have
\begin{align*}
    \mathbb{E}[\|x_{k+1}-x^*\|^2] &\leq \mathbb{E}[\|x_k-x^*\|^2] + t_k^2\mathbb{E}[\|g_k\|^2] \\
    & -2t_k\langle \mathbb{E}[g_k], x_k-x^* \rangle
\end{align*}
From Lemma 1 in \cite{fan2021xorsgd}, we know $-\langle \mathbb{E}[g_k], x_k-x^* \rangle\leq -\frac{1}{c}\langle \nabla f(x_k), x_k-x^*\rangle$. Then we have
\begin{align*}
    \mathbb{E}[\|x_{k+1}-x^*\|^2] &\leq \mathbb{E}[\|x_k-x^*\|^2] + t_k^2\mathbb{E}[\|g_k\|^2]\\ &- \frac{2t_k}{c}\langle \nabla f(x_k), x_k-x^*\rangle\\
    &\leq \mathbb{E}[\|x_k-x^*\|^2] + t_k^2\mathbb{E}[\|g_k\|^2]\\
    &-\frac{2t_k}{c}[f(x_k)-f(x^*) + \frac{\mu}{2}\|x_k-x^*\|^2]
\end{align*}
The last inequality is because $f$ is $\mu$-strongly convex. After some rearranging and taking expectation on both sides, we have
\begin{equation*}
\begin{aligned}
    &\mathbb{E}[f(x_k)] - f(x^*) \leq \frac{ct_k}{2}\mathbb{E}[\|g_k\|^2]\|]\\ 
    & ~~ +\frac{c-t_k\mu}{2t}\mathbb{E}[\|x_{k}-x^*\|^2] 
    - \frac{c}{2t_k} \mathbb{E}[\|x_{k+1}-x^*\|^2].
\end{aligned} 
\end{equation*}
We know $\mathbb{E}[\|g_k\|^2]=\|\mathbb{E}[g_k]\|^2+\text{Var}(g_k)$, and from Lemma 1 in \cite{fan2021xorsgd} we have
\begin{align*}
\|\mathbb{E}[g_k]\|^2\leq & c\langle \nabla f(x_k), \mathbb{E}[g_k]\rangle \\
 & \leq  c\|\nabla f(x_k)\|\|\mathbb{E}[g_k]\|\leq cL\|\mathbb{E}[g_k]\|
\end{align*}
Therefore, $\|\mathbb{E}[g_k]\|^2\leq c^2L^2$ is bounded above and $\mathbb{E}[\|g_k\|^2]\leq c^2L^2+\sigma^2$. Apply this bound, we have
\begin{equation}\label{key_ineq}
\begin{aligned}
    &\mathbb{E}[f(x_k)] - f(x^*)\\
    &\leq \frac{ct_k(c^2L^2+\sigma^2)}{2} +\frac{c-t_k\mu}{2t_k}\mathbb{E}[\|x_{k}-x^*\|^2]\\
    &- \frac{c}{2t_k} \mathbb{E}[\|x_{k+1}-x^*\|^2]
\end{aligned} 
\end{equation}
With $t_k=\frac{c}{\mu k}$, the above inequality becomes
\begin{align*}
    &\mathbb{E}[f(x_k)] - f(x^*) 
    \leq \frac{c^2(c^2L^2+\sigma^2)}{2\mu k} \\ &+\frac{\mu(k-1)}{2}\mathbb{E}[\|x_{k}-x^*\|^2] 
    - \frac{\mu k}{2} \mathbb{E}[\|x_{k+1}-x^*\|^2].
\end{align*} 
Summing the above equations for $k=1,\ldots,K$, we get
\begin{align*}
    &\mathbb{E}[f(\frac{1}{K}\sum_{k=1}^{K} x_k)] - f(x^*) \leq \frac{1}{K}\sum_{k=1}^{K}\mathbb{E}[f(x_k)] - f(x^*)\\
    &\leq \frac{B}{2\mu K}\sum_{k=1}^{K} \frac{1}{k} + \frac{\mu}{2 K} [0-K\mathbb{E}[\|x_{K+1}-x^*\|^2]\\
    & \leq \frac{B}{2\mu K} (1+\log(K))
\end{align*}
where $B = c^2(c^2L^2+\sigma^2)$. The first inequality is obtained using the convexity of $f$; the second inequality  is obtained from a telescoping sum. Let $\overline{x_K}=\frac{1}{K}\sum_{k=1}^{K} x_k$, the above inequality can be written as
\begin{align*}
\mathbb{E}[f(\overline{x_K})] - f(x^*) \leq \frac{B}{2\mu K} (1+\log(K))
\end{align*}
Therefore, the convergence rate is $O(\frac{\log(K)}{K})$. This completes the proof. 
\end{proof}

Theorem \ref{Th:XOR_SGD_bound_constrained} states that by choosing the step size $t_k$ inverse proportional to the step $k$, i.e., $t_k=\frac{c}{\mu k}$, our XOR-PGD can converge to the optimal value in $O(\frac{\log K}{K})$. In practice, we can increase the sample size from 1 to $N$ to further reduce the variance. By replacing the objective $f(x)$ in Theorem \ref{Th:XOR_SGD_bound_constrained} with $\mathbb{E}_{\theta\sim p(\theta)}f(x,\theta)$, and noticing $Var(\overline{g_k}) = Var_\theta(\nabla_x f(x, \theta)) / N$ due to the sample size $N$, we have the following Theorem \ref{Th:main}.
\begin{Th}\label{Th:main}
Let $\rho,\kappa$ be the constant approximation factor as in \cite{fan2021xorsgd}, function $f(x,\theta): \mathbb{R}^d\times\{0,1\}^n\rightarrow \mathbb{R}$ be a $L$-smooth and $\mu$-strongly convex function w.r.t. $x$. Denote $OPT=\min_{x\in C} \mathbb{E}_{\theta\sim Pr(\theta)}f(x,\theta)$ as the global optimum in constraints set $C$.
Let $\sigma^2=\max_x\{Var(\nabla_x f(x,\theta))\}$ and $\varepsilon^2=\max_x\{||\mathbb{E}[\nabla_x f(x, \theta)]||_2^2\}$. 
For any $1\leq\rho\kappa\leq\sqrt{2}$, step size $t\leq \frac{\rho\kappa}{\mu k}$ and sample size $N\geq1$, $\overline{x_K}=\frac{1}{K}\sum_{k=1}^{K} x_k$ is the output of XOR-PGD (Algorithm \ref{alg: XOR-PGD}) and $\mbox{obj} = \mathbb{E}_{\theta}[f(\overline{x_K}, \theta)]$ is the objective function value at $\overline{x_K}$. We have: 
\begin{align*}
    &\mathbb{E}_{\overline{x_K}}[\mbox{obj}] - OPT\\
     &\leq \left (\frac{\rho^4\kappa^4L^2}{2\mu}+\frac{\rho\kappa(\sigma^2+\varepsilon^2)}{2\mu N}\right )\frac{1+\log(K)}{K} .
\end{align*}
\end{Th}
\begin{proof}(Theorem 5)
Since we use $N$ samples at each iteration, we have $\overline{g_k}=\frac{1}{N}\sum_{i=1}^Ng_k^i$ and $\mathbb{E}[\overline{g_k}]=\mathbb{E}[g_k^i]$.
In each iteration $k$ we can adjust the parameters in XOR-Sampling to make the tail $\epsilon\eta_{\phi}$ zero, then for each sample $g_k^i$ we can obtain from Theorem 2 that
\begin{align}
    \frac{1}{\rho\kappa} \mathbb{E}_\theta [\nabla f(x_k, \theta)]^+ &\leq \mathbb{E}[g_k^{i+}]\leq \rho\kappa\mathbb{E}_\theta [\nabla f(x_k, \theta)]^+.\label{eq:pos2}\\
    \rho\kappa \mathbb{E}_\theta [\nabla f(x_k, \theta)]^- &\leq \mathbb{E}[g_k^{i-}]\leq  \frac{1}{\rho\kappa}\mathbb{E}_\theta [\nabla f(x_k, \theta)]^-.\label{eq:neg2}
\end{align}
The variance of each sample $g_k^i$ can also be bounded by
\begin{align*}
    &Var(g_k^i) \\
    &= \mathbb{E}_{\theta' \sim p'(\theta')} [||\nabla f(x_k, \theta')||_2^2] - ||\mathbb{E}_{\theta' \sim p'(\theta')} [\nabla f(x_k, \theta')]||_2^2,\\
    &\leq \rho\kappa\mathbb{E}_{\theta \sim p(\theta)} [||\nabla f(x_k, \theta)||_2^2],\\
    &= \rho\kappa (Var(\nabla f(x_k,\theta)) + ||\mathbb{E}_{\theta \sim p(\theta)} [\nabla f(x_k, \theta)]||_2^2),\\
    &\leq \rho\kappa(\sigma^2 +\varepsilon^2).
\end{align*}
Denote $\overline{g_k}^+=\max\{\overline{g_k},\textbf{0}\}$ and $\overline{g_k}^-=\min\{\overline{g_k},\textbf{0}\}$. 
Clearly, $g_k^{i+} \geq 0$ and $g_k^{i-} \leq 0$. Moreover,
for a given dimension, either $g_k^{i+}=0$ for that dimension or $g_k^{i-}=0$. 
Evaluating $\overline{g_k}$ dimension by dimension, we can see that
$\overline{g_k}^+=\frac{1}{N}\sum_{i=1}^N g_k^{i+}$ and $\overline{g_k}^-=\frac{1}{N}\sum_{i=1}^N g_k^{i-}$.
Combined with Equation~\ref{eq:pos2} and \ref{eq:neg2}, we know 
\begin{align*}
    \frac{1}{\rho\kappa} \mathbb{E}_\theta [\nabla f(x_k, \theta)]^+ &\leq \mathbb{E}[\overline{g_k}^+]\leq \rho\kappa\mathbb{E}_\theta [\nabla f(x_k, \theta)]^+.\\
    \rho\kappa \mathbb{E}_\theta [\nabla f(x_k, \theta)]^- &\leq \mathbb{E}[\overline{g_k}^-]\leq  \frac{1}{\rho\kappa}\mathbb{E}_\theta [\nabla f(x_k, \theta)]^-.
\end{align*}

Because $\mathbb{E}[\overline{g_k}]=\mathbb{E}[g_k^i]$, we also have
\begin{align*}
    Var(\overline{g_k})=\frac{1}{N^2}Var(\sum_{i=1}^Ng_k^i)=\frac{Var(g_k^i)}{N}.
\end{align*}
Then the variance of $\overline{g_k}$ can be bounded as
\begin{align*}
    Var(\overline{g_k})&\leq \frac{\rho\kappa(\sigma^2+\varepsilon^2)}{N}.
\end{align*}

Therefore, we can then apply Theorem 5 to get the result in equation 5.
\begin{align*}
     &\mathbb{E}_{\overline{x_K}}[\mathbb{E}_{\theta}[f(\overline{x_K}, \theta)]]  -\mathbb{E}_{\theta}[f(x^*,\theta)]\\
     &\leq \frac{\rho^4\kappa^4L^2+\rho^2\kappa^2\max_k\{Var(\overline{g_k})\}}{2\mu K} (1+\log(K))\\
     &\leq \frac{\rho^4\kappa^4L^2+\rho^2\kappa^2\frac{\sigma^2+\varepsilon^2}{N}}{2\mu K} (1+\log(K))
\end{align*}
which can also be written as
\begin{align*}
&\mathbb{E}_{\overline{x_K}}[obj] - OPT\\
     &\leq \left (\frac{\rho^4\kappa^4L^2}{2\mu}+\frac{\rho\kappa(\sigma^2+\varepsilon^2)}{2\mu N}\right )\frac{1+\log(K)}{K} .
\end{align*}
This completes the proof.
\end{proof}

In Theorem \ref{Th:main} we can see the difference to the optimum is inversely proportional to both $K$ and $N$ to some extent. To tighten the bound with fixed number of iterations $K$, we can either conduct more XOR-Sampling scheme leading to smaller $\rho\kappa$, or generate more samples at each iteration to reduce the variance term.

\subsection{Improved XOR-PGD}
Although we prove that the proposed XOR-PGD has a theoretical guarantee on the convergence rate, this convergence rate of $O(\frac{\log K}{K})$ is not rather surprising. Therefore, we further improve the algorithm by selecting a more reasonable step size. By selecting $t_k=\frac{2\rho\kappa}{\mu(k+1)}$ as in Algorithm \ref{alg: improved XOR-PGD}, and compute the last output from a weigthed average ${x_K}\gets \reallywidehat{x_K}= \frac{2\sum_{k=1}^K k x_k}{K(K+1)} $,  we develop a improved algorithm of XOR-PGD named (I)XOR-PGD in Algorithm \ref{alg: improved XOR-PGD}. We prove that this (I)XOR-PGD is able to further accelerate the convergence rate to $O(\frac{1}{K})$. 

\begin{Th}\label{Th:XOR_SGD_bound_constrained_improved} 
Let $f:\mathbb{R}^d\rightarrow \mathbb{R}$ be a $L$-smooth and $\mu$-strongly convex function and $x^*=\arg\min_{x\in C} f(x)$, where $C$ is a convex  set of $x$. In iteration $k$, $g_k$ is the estimated gradient, i.e., $x_{k+1}=\mathbb{P}_C(x_{k}-t_kg_k)$ where $Var(g_k)\leq \sigma^2$ and $\mathbb{P}_C(x):=\arg \min_{y\in C}\frac{1}{2}\|x-y\|_2^2$ is the projection of $x$ onto $C$. If $t_k = \frac{2c}{\mu (k+1)}$ and there exists $1\leq c\leq\sqrt{2}$ s.t. $\frac{1}{c}[\nabla f(x_k)]^+ \leq \mathbb{E}[g_k^+]\leq c[\nabla f(x_k)]^+$ and $c[\nabla f(x_k)]^- \leq \mathbb{E}[g_k^-]\leq \frac{1}{c}[\nabla f(x_k)]^-$, then the convergence rate of the Improved XOR-PGD algorithm is $O(\frac{1}{K})$:
\begin{align*}
    &\mathbb{E}\left [f\left ( \reallywidehat{x_K}\right )\right ]-f(x^*) \leq \frac{2B}{\mu (K+1)},
\end{align*}
where $\reallywidehat{x_K}=\frac{2}{K(K+1)} \sum_{k=1}^K k x_k$.
\end{Th}

\begin{proof}
With $t_k=\frac{2c}{\mu (k+1)}$ and multiplying~\eqref{key_ineq} by $k+1$, we have 
\begin{align*}
    &k(\mathbb{E}[f(x_k)]-f(x^*))\leq \frac{kB}{\mu (k+1)}\\& + \frac{\mu}{4} \left ( k(k-1)\mathbb{E}[\|x_{k}-x^*\|^2] - k(k+1)\mathbb{E}[\|x_{k+1}-x^*\|^2]\right )\\
    &\leq \frac{B}{\mu}  + \frac{\mu}{4}  ( k(k-1)\mathbb{E}[\|x_{k}-x^*\|^2]\\
    & \qquad ~ - k(k+1)\mathbb{E}[\|x_{k+1}-x^*\|^2] )
\end{align*}
Summing from $k=1$ to $k=K$, we have
\begin{align*}
    &\sum_{k=1}^K k(\mathbb{E}[f(x_k)]-f(x^*))\\ &\leq \frac{KB}{\mu} + \frac{\mu}{4}(0-K(K+1)\mathbb{E}[\|x_{k+1}-x^*\|^2]).
\end{align*}
Then, we have
\begin{align*}
    &\mathbb{E}\left [f\left ( \frac{2}{K(K+1)} \sum_{k=1}^K k x_k\right )\right ]-f(x^*) +\frac{\mu}{2} \mathbb{E}[\|x_{k+1}-x^*\|^2]\\ &\leq \frac{2B}{\mu (K+1)},
\end{align*}
which implies
\begin{align*}
    &\mathbb{E}\left [f\left ( \reallywidehat{x_K}\right )\right ]-f(x^*) \leq \frac{2B}{\mu (K+1)},
\end{align*}
where $\reallywidehat{x_K}=\frac{2}{K(K+1)} \sum_{k=1}^K k x_k$. Here we prove the $O(1/K)$ convergence rate. 
\end{proof}

Similarly, we can get the final convergence rate of the Improved XOR-PGD as in Algorithm \ref{alg: improved XOR-PGD} by increasing the sample size from 1 to $N$.

\begin{Co}\label{improved main}
(Main) Let $\rho,\kappa$ be as before, function $f(x,\theta): \mathbb{R}^d\times\{0,1\}^n\rightarrow \mathbb{R}$ be a $L$-Lipschitz and $\mu$-strongly convex function w.r.t. $x$. Denote $OPT=\min_{x\in C} \mathbb{E}_{\theta\sim Pr(\theta)}f(x,\theta)$ as the global optimum. 
Let $\sigma^2=\max_x\{Var(\nabla_x f(x,\theta))\}$ and $\varepsilon^2=\max_x\{||\mathbb{E}[\nabla_x f(x, \theta)]||_2^2\}$. 
For any $1\leq\rho\kappa\leq\sqrt{2}$, step size $t\leq \frac{2\rho\kappa}{\mu (k+1)}$ and sample size $N\geq1$, $\reallywidehat{x_K}=\frac{2}{K(K+1)} \sum_{k=1}^K k x_k$ is the output of Improved XOR-PGD (algorithm \ref{alg: improved XOR-PGD}) and $\mbox{obj} = \mathbb{E}_{\theta}[f(\reallywidehat{x_K}, \theta)]$ is the objective function value at $\reallywidehat{x_K}$. We have: 
\begin{align*}
    &\mathbb{E}_{\reallywidehat{x_K}}[\mbox{obj}] - OPT\\
     &\leq \left (\frac{\rho^4\kappa^4L^2}{\mu}+\frac{\rho\kappa(\sigma^2+\varepsilon^2)}{\mu N}\right )\frac{1}{K+1} .
\end{align*}
\end{Co}

We skip the proof since it is very similar to the proof of Theorem \ref{Th:main}. It should be noticed that although hard to compute, $\sigma^2$ and $\varepsilon^2$ are from the input which do not depend on the algorithm. In addition, it should be noticed that the convergence rate of our algorithm is determined by the approximation constant $\rho\kappa$ from XOR-sampling. The smaller this constant is, the lower the success rate of XOR-Sampling has. By setting proper parameter values, we can get $\rho\kappa=\sqrt{2}$. As a consequence, we can collect $N$ samples successfully by running XOR-Sampling around $40N$ times. The time complexity can be further reduced via parallel sampling. In addition, the sampling procedure is independent of the optimization step since $Pr(\theta)$ does not depend on $x$.

%% file: tex/exp.tex
We evaluate our (Improved) XOR-PGD algorithm on the same benchmark of XOR-SGD \citep{fan2021xorsgd}, the inventory management (\cite{ziukov2016literature,shapiro2007tutorial}) and the network design problems (\cite{sheldon2012maximizing,WuXSG17XORSampling,xiaojian2016}). For comparison, we consider XOR-SGD, and also those that use SGD methods with Gibbs Sampling, Belief Propagation (BP) (\cite{yedidia2001generalized,murphy2013loopy}), or Belief Propagation Chain (BPChain) (\cite{fan2020BPChainCD}). 
Similar to the setting in \cite{fan2021xorsgd}, for each setting of both applications, to produce a sample, we let Gibbs sampling have 100 burn in samples, and then draws one sample every 30 steps. We fix the number of iteration steps of both BP and BPChain as $20$, which is enough for belief propagation to converge.
We allow SGD with Gibbs sampling, BP and BPChain to draw more samples than both XOR-PGD and XOR-SGD for a fair comparison. For both applications, we use MRF as probabilistic models for $Pr(\theta)$. 
All experiments were conducted using single core architectures on Intel Xeon Gold 6126 2.60GHz machines with 96GB RAM and a wall-time limit of $10$ hours.
We use IBM ILOG Cplex 12.71 as the solver of NP oracle to produce each XOR sample for both XOR-PGD and XOR-SGD. Notice that the projection step in XOR-PGD is also sovled by Cplex. 
Once a solution $x$ is generated by 
either algorithm, we use an exact weighted counter ACE \cite{barton2016ace} to evaluate $\mathbb{E}_{\theta\sim Pr(\theta)} f(x, \theta)$ exactly. All objective values reported are from ACE.
Since we run each algorithm on one single core with a wall-time limit of $10$ hours for a fair comparison thus not all algorithms can complete all iterations, we report the best results found by each algorithm within the time limit.

\begin{table*}[ht]
    \centering
    \small
    \begin{tabular}{l|cccccccccc}
    \toprule
    & \multicolumn{10}{c}{Number of materials} \\
      \hline
      & 10 & 20  & 30 & 40 & 50 & 60 & 70 & 80 & 90 & 100 \\
      \hline
        savings/Gibbs &  8.57\% & 19.29\% & 13.43\% & 13.45\% & 14.48\% & 13.46\% & 11.6\% & 13.87\% & 17.31\% & 20.32\% \\
        savings/BP   &  22.89\% & 20.7\% & 17.21\% & 18.63\% & 19.52\% & 16.94\% & 14.94\% & 25.16\% & 30.7\% & 34.47\%  \\
        savings/BPChain & 6.57\% & 13.74\% & 12.17\% & 12.01\% & 9.19\% & 14.3\% & 12.71\% & 19.2\% & 24.89\% & 30.83\% \\
        savings/XOR-SGD & 1.54\% & 1.74\% & 5.31\% & 2.86\% & 1.21\% & 2.64\% & 2.24\% & 2.87\% & 5.38\% & 10.27\% \\
    \bottomrule
    \end{tabular}
    \caption{The percentage of savings of the solutions found by Improved XOR-PGD ((I)XOR-SGD) in Algorithm \ref{alg: improved XOR-PGD} against other methods on 100\% storage limit varying the number of materials. We can see the Improved XOR-PGD on average saves 10\% cost against the first three methods and 2\% against XOR-SGD, a naive method coupled with XOR-Sampler.}
    \label{tab:saving}
    \vspace{0.1in}
\end{table*}

\begin{table*}[ht]
    \centering
    \small
    \begin{tabular}{l|cccccccccc}
    \toprule
    & \multicolumn{10}{c}{Number of materials} \\
      \hline
      & 10 & 20  & 30 & 40 & 50 & 60 & 70 & 80 & 90 & 100 \\
      \hline
        XOR-SGD   &  98.3\% & 97.3\% & 95.3\% & 94.1\% & 93.4\% & 91.6\% & 89.6\% & 87.3\% & 83.4\% & 80.9\%  \\
        XOR-PGD & 99.4\% & 99.3\% & \textbf{98.5\%} & \textbf{98.1\%} & 97.4\% & \textbf{96.6\%} & 95.2\% & 93.4\% & \textbf{92.9\%} & 91.3\% \\
        (I)XOR-PGD & \textbf{99.6\%} & \textbf{99.4\%} & 98.3\% & 97.6\% & \textbf{97.5\%} & 96.4\% & \textbf{95.9\%} & \textbf{93.5\%} & 91.5\% & \textbf{91.7\%} \\
    \bottomrule
    \end{tabular}
    \caption{The average rate of constraints satisfaction of 1000 solutions found by both  XOR-PGD and (I)XOR-SGD against XOR-SGD on 100\% storage limit and with 60 samples each iteration varying the number of materials. We can see solutions found by both of our methods have more than $90\%$ rate of constraints satisfaction even when the searching space is very large, which is $10\%$ more than that from XOR-SGD.}
    \label{tab:constraints}
    \vspace{0.1in}
\end{table*}

\begin{figure*}[t]
\centering
\subfigure{\label{fig:inventory_budget}
\includegraphics[width=0.35\linewidth]{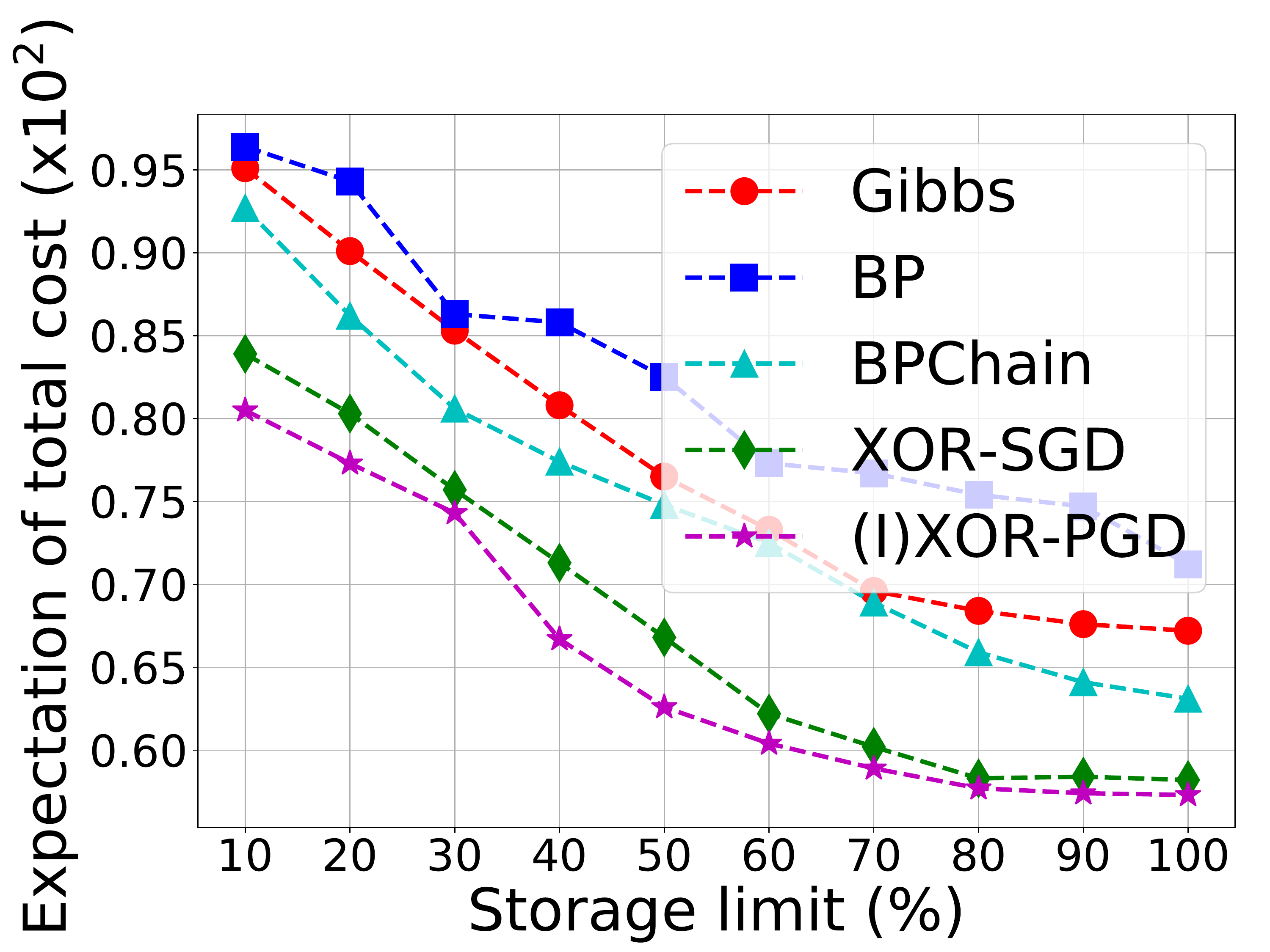}}
\subfigure{\label{fig:inventory_sample}
\includegraphics[width=0.35\linewidth]{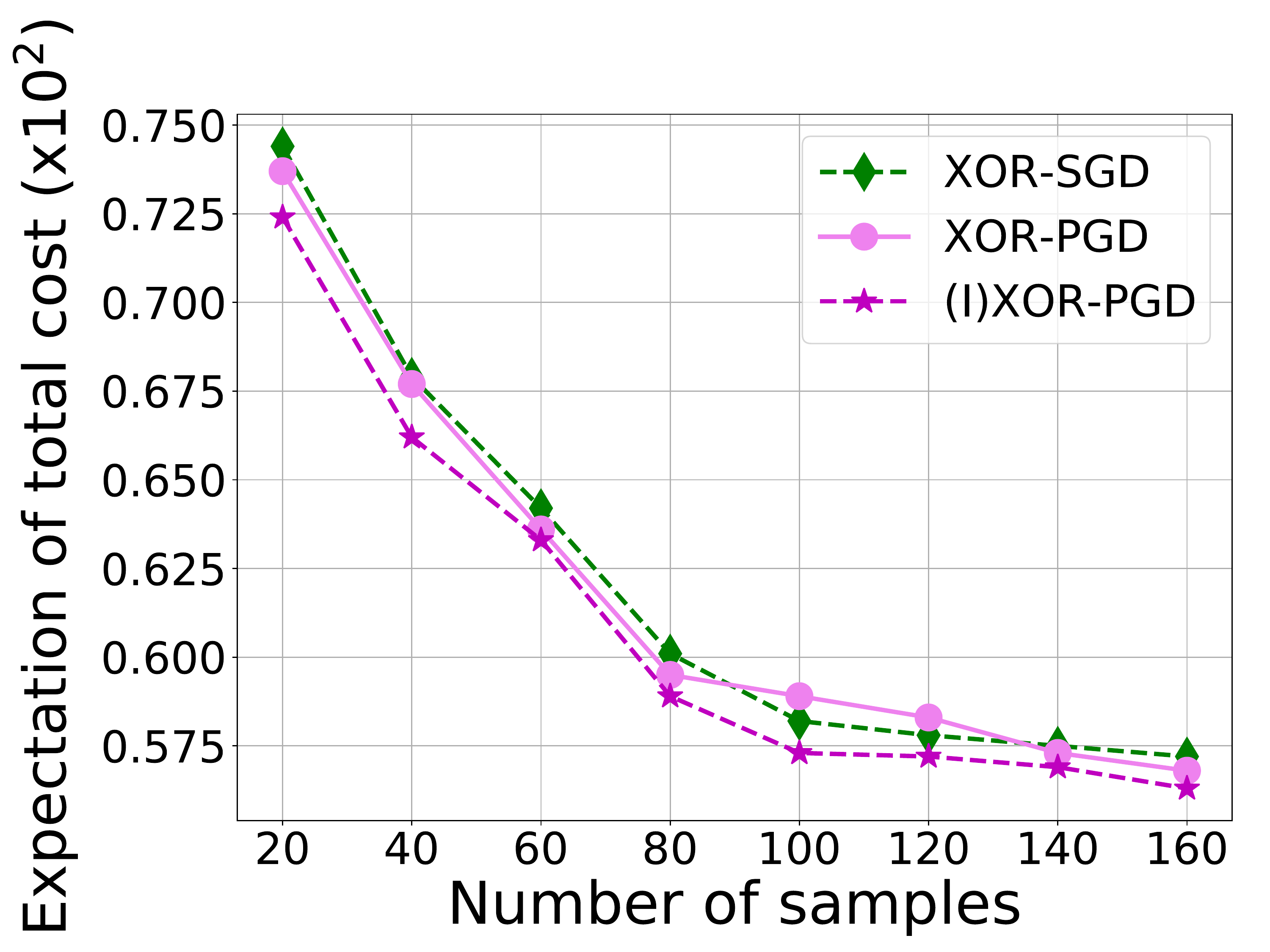}}
\vspace{0.1in}
\subfigure{\label{fig:bar_network}
\includegraphics[width=0.32\linewidth]{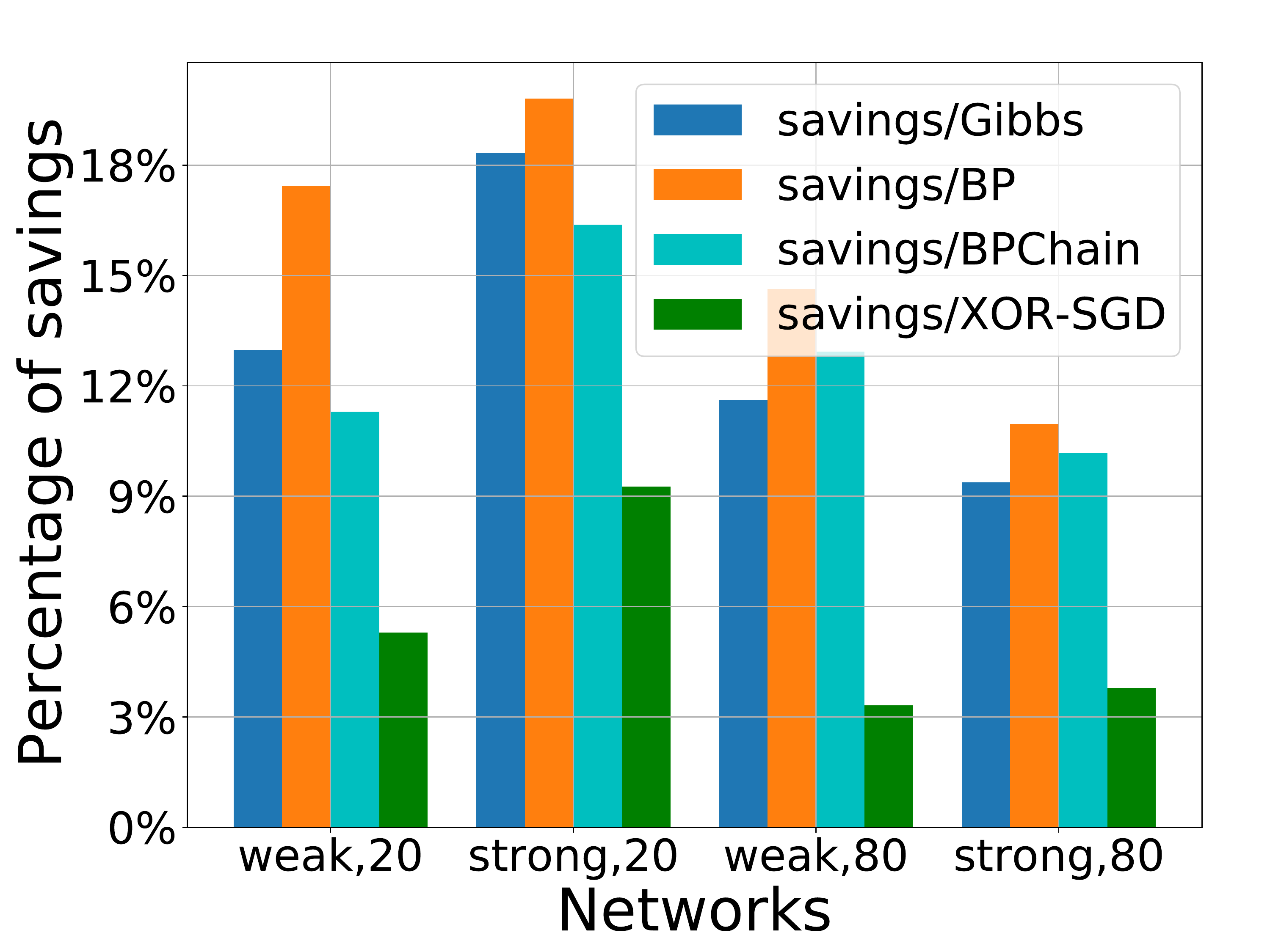}}
\subfigure{\label{fig:network_budget}
\includegraphics[width=0.32\linewidth]{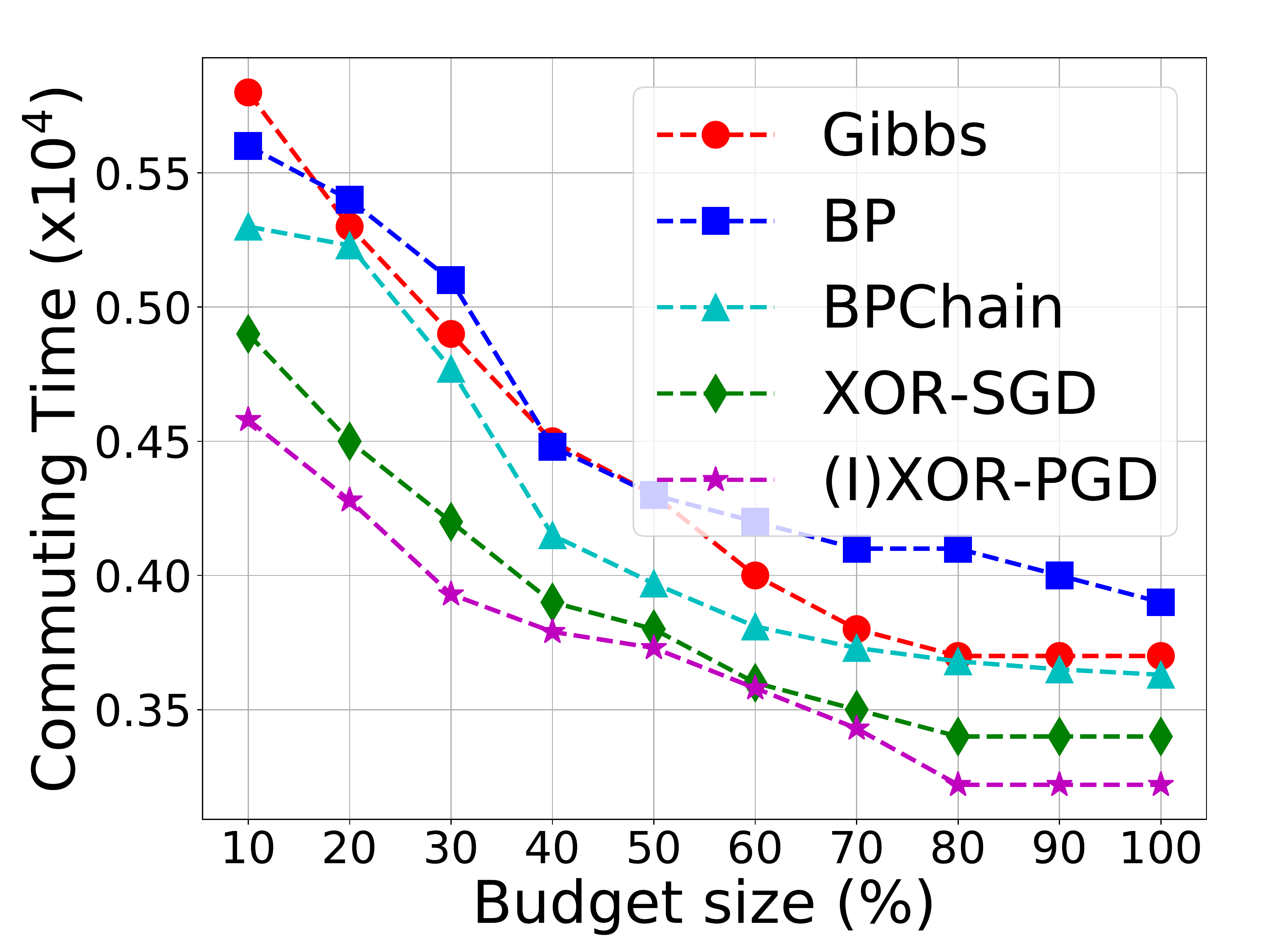}}
\subfigure{\label{fig:network_sample}
\includegraphics[width=0.32\linewidth]{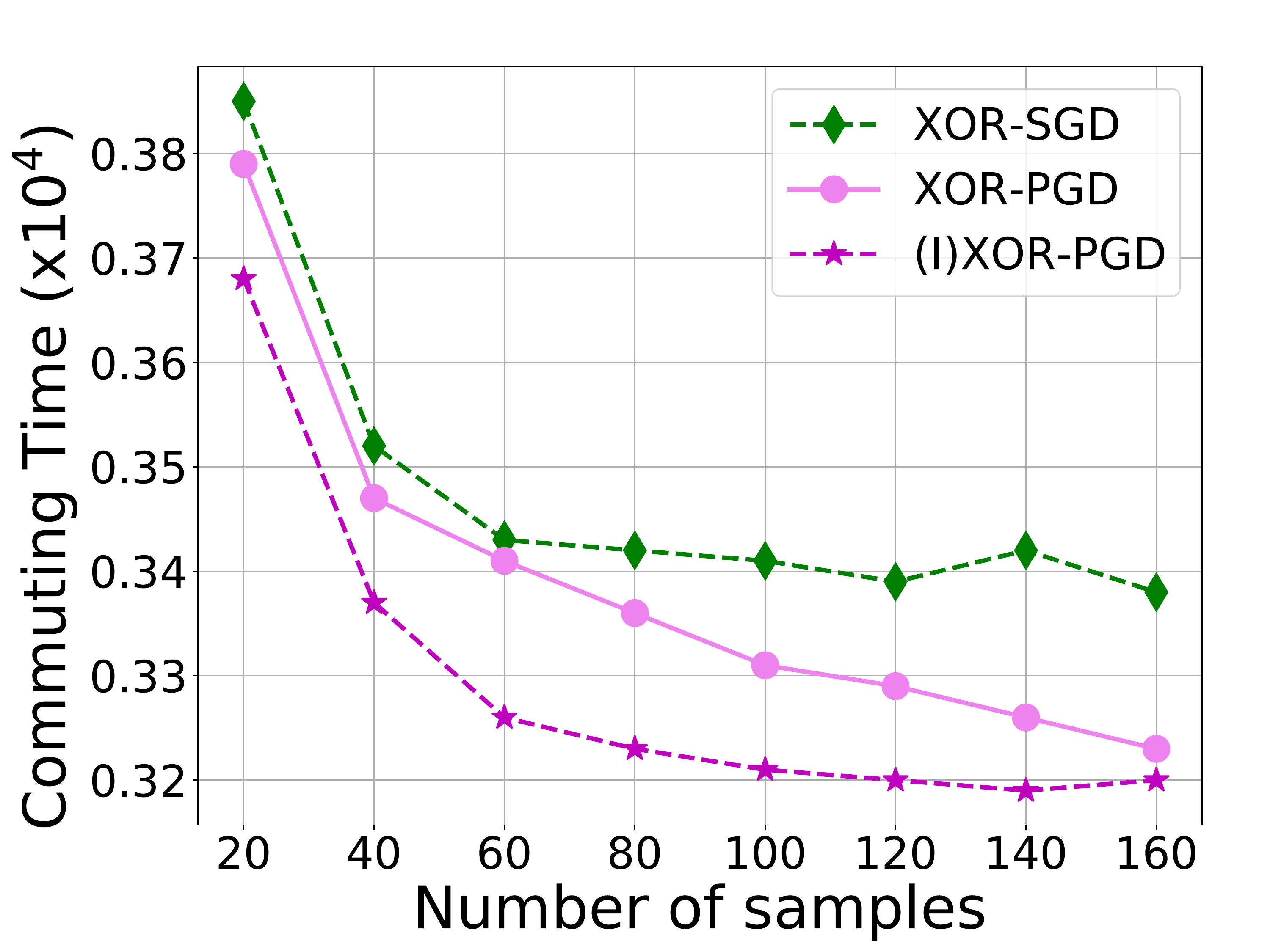}}
\caption{ The upper row is experimental results on the inventory management problem and the bottom row is the results on the network design problem. The Improved XOR-PGD ((I)XOR-PGD) is better than the baselines in all cases. \textbf{(Upper left)} The objective values found by all methods with 50 materials varying storage limits. our method always finds better solutions than the other approaches. \textbf{(Upper right)} The objective values found by all methods with different number of samples for approximation in each iteration with 100\% storage limit and 50 materials. \textbf{(Bottom left)} The percentage of savings of (I)XOR-PGD against other methods with  100\% budget. (I)XOR-PGD saves on average $> 3\%$  commuting time on all 4 benchmarks.  \textbf{(Bottom middle)} Commuting time of the solutions found by (I)XOR-PGD and baselines varying budgets on network ``weak 20''. (I)XOR-PGD always leads to the least commuting time no matter how much the budget is given. \textbf{(Bottom right)} Commuting time found by the three methods with different number of samples (100\% budget, ``weak 20''). Both XOR-PGD and (I) XOR-PGD outperforms XOR-SGD with less XOR samples.}
\vspace*{0.1in}
\label{fig:exp}
\end{figure*}

\subsection{Stochastic Inventory Management}\label{subsec:inventory}
We first investigate our algorithm on the  stochastic inventory management problem studied in \cite{shapiro2007tutorial}.
Assuming there are $n$ materials. The demand of material $i$ is $d_i$. Let $d=(d_1, \ldots, d_n)^T$
be the demand vector. 
The manager stocks $x_i$ amount of material $i$ at the beginning of the season. Each unit of material $i$ takes storage space $w_i$, and the total amount of pre-order is limited by the available storage space $X$. 
At the end of the production season, demand $d$ will be revealed to
the the manager.
We assume the cost of ordering the $i$-th material is $c_i$ per unit. If the demand $d_i>x_i$, then a back order is needed, of which one unit costs $b_i\geq c_i$. Overall, the cost for back order is $b_i(d_i-x_i)$ if $d_i > x_i$, and is zero otherwise. On the other hand, if $d_i < x_i$, then a holding cost of $h_i$ per unit is incurred, leading to an additional total cost $h_i(x_i-d_i)$. Summing it up, the cost for material $i$ is $G_i = c_ix_i + b_i[d_i-x_i]^+ + h_i[x_i-d_i]^+$ where $[a]^+$ denotes the maximum of $a$ and 0. Then,
the total cost will be $G(x,d)=\sum_{i=1}^{n}G_i$. The manager want to minimize his operational cost, which translates to this problem:
\begin{align}\label{eq:inventory}
    \min_{x\geq 0}~~\mathbb{E}_{d \sim Pr(d)}[G(x,d)], \quad s.t. \quad w^T x \leq X.
\end{align}
where $G(x,d)$ is convex w.r.t. $x$.
We run the experiments varying the number of materials $n$, 
the storage limit, and the number of samples we use in XOR-PGD and other methods. Parameters and experimental settings are the same as in \cite{fan2021xorsgd}.
In terms of the parameters in XOR-Sampling we fix $P=100, b=7, \epsilon=0.01$ and the others the same as in \cite{ermon2013embed} to guarantee $\rho\kappa=\sqrt{2}$. Learning rate $t$ is $0.1$ at first and divided by $10$ after $50$ iterations, then further divided by $10$ after $100$ iterations. $\eta$ is 10 at first and divided by $10$ after $50$ iterations, then further divided by $10$ after $100$ iterations. The total number of both $K$ and $M$ are set to be $200$.
However, since we run each algorithm on one single core with a wall-time limit of $10$ hours for a fair comparison, not all algorithms can complete all iterations. The plots are based on the best results found by each algorithm within the time limit.

Table \ref{tab:saving} shows that our algorithm outperforms other methods varying number of materials on the percentage reduction of the objective values of the solutions. In math form, for example for Gibbs Sampling, the metric is $(obj(\mbox{Gibbs}) - obj(\mbox{(I)XOR-PGD})) / obj(\mbox{Gibbs})$ (metrics for other approaches are analogous). The objective optimized by (I)XOR-PGD is on average 2 percent better than XOR-SGD and $10\%$ more than that optimized by other baselines. In addition, Table \ref{tab:constraints} shows that the rate of constraints satisfaction of the solutions from either XOR-PGD or (I) XOR-PGD is $10\%$ more than that of XOR-SGD even in very large searching space.
The first row in Figure \ref{fig:exp} shows the objective values of solutions varying the storage limit, and the objective values varying the number of samples with $100\%$ storage limit and 50 different materials. 
We can see from the left figure that with the storage limit increasing, (I)XOR-PGD is always better than all baselines, and from the right figure that (I)XOR-PGD found better solutions with less samples than XOR-SGD. Since \cite{fan2021xorsgd} already reported that XOR-SGD is faster than the other competing approaches, here by comparing efficient sample size with XOR-SGD we conclude that (I)XOR-PGD is faster than all the other methods. 

\subsection{Stochastic Network Design}\label{subsec:network}
We then consider the expected commuting time of a random walk
in network optimization  \cite{ghosh2008minimizing,McClure2016ConnectingMT,Inman2013DevelopingPF}, which is also a benchmark in \cite{fan2021xorsgd}.
Given an undirected graph $G=(V,E)$, where $|V|=m, |E|=n$.
Let $g = (g_1, \dots, g_n)^T$, and edge $e\in E$ is associated with a binary random variable $\theta_e$ that describes the state of the edge during disasters. $\theta_e=0$ means that the edge is destroyed, and 1 otherwise. Let $\theta=(\theta_1, \dots, \theta_n)$.
It will take money $c_e$ to increase one unit of $g_e$ and we have a total budget of $B$. 
Then, From \cite{ghosh2008minimizing} we know the commuting time
\begin{align*}
    \overline{C}(g, \theta)=\frac{4(\mathbf{1}^Tg)}{(m-1)}\Big{(}\mathbf{Tr}(L+\mathbf{1}\mathbf{1}^T/m)^{-1}-1\Big{)}
\end{align*}
which is convex w.r.t. $g$, where $L$ is the weighted Laplacian matrix.
We would like to find the best network improvement plan under the given budget, which minimizes the expected commuting time averaged over all stochastic events to maximizes the network connectivity. Mathematically, our problem can be formulated as the following problem: 
\begin{align*}
    \min_{\Delta g\geq 0}~\mathbb{E}_{\theta\sim Pr(\theta)}[\overline{C}(g+\Delta g, \theta)], 
    ~ s.t. ~ \sum_{e\in E} c_e \Delta g_e\leq B.
\end{align*}
We evaluate our algorithms on the Flood Preparation problem for the emergency medical services (EMS) on road networks as studied in \cite{xiaojian2016} and test our algorithm on four benchmarks involving the weak and the strong network originally evaluated in \cite{fan2021xorsgd} varying the percentage of the largest budget size and samples size, which is shown in the bottom row of Figure \ref{fig:exp}. The largest budget size $B$ is $1000$. Total number of SGD iterations is 2000, while not all algorithms can complete all 2000 iterations within the time limit of 10 hours. Parameters in XOR-Sampling are set to be the same as in \cite{fan2021xorsgd}. 

The results are similar to those for the inventory management problem, where results clearly show that (I)XOR-PGD outperforms other methods both in efficiency and in the quality of solutions. 
On all of the four different type of networks (I)XOR-PGD saves on average $>3\%$ against XOR-SGD and more than $9\%$ against the other competing methods.
In addition, we would like to emphasize that (I)XOR-PGD with 40 samples already outperforms XOR-SGD with 60 samples in the right figure while finds better solutions at the same time, and even naive XOR-PGD finds better solutions than XOR-SGD given the same sample size.

The left figure in Figure 3 shows the percentage of savings between SGD with other sampling methods and XOR-SGD among all of the 4 different networks, while the middle and the right figures show the averaged commuting time with regard to different budget sizes and different number of samples, respectively. For the left and the middle figures, we let XOR-SGD take 100 samples in each iteration while SGD with other methods take 10,000. We can see from the left figure that objective optimized by XOR-SGD is at least 5\% better than that optimized by other methods for all the 4 different networks. In addition, from the middle and the right figures we know that with the increase of either budget size or the number of samples, our method can find consistently better solutions than the compared methods. In particular, from the right figure we can see even 40 samples in each iteration are enough for XOR-SGD to compete with the result from Gibbs with 20,000 samples. Meanwhile, XOR-SGD also runs faster than the compared method under this situation. In this experiment, XOR-SGD with 40 samples take 1 minutes 40 seconds per SGD iteration, while SGD with 20,000 Gibbs samples need 2.5 minutes per iteration. 
Since sampling time of both BP and BPChain is no shorter than Gibbs Sampling, we thus conclude that XOR-SGD outperforms other methods both in efficiency and in the quality of solutions found. 

%% file: tex/conclusion.tex
We proposed XOR-PGD, a provable algorithm to attack constrained convex stochastic optimization problems, which are crucial for many decision-making applications with uncertainty. 
We showed theoretically that our algorithm has a linear convergence rate to the global optimum by chossing proper step sizes. Empirically, we demonstrated the superior performance of XOR-PGD on both the stochastic inventory management and the stochastic network design problems. 
In particular, 90\% solutions obtained by XOR-PGD satisfy the constraints set $C$ even when the searching space is very large in the inventory management problem, approximately 10\% more that of competing methods. Besides, XOR-PGD converges faster than XOR-SGD by accessing 20\% less XOR samples in each iteration and is able to find better solutions in the stochastic network design problem.
Overall, our paper demonstrates the power of integrating
cutting-edge computer science technology with real-world  problems.
Our paper will also stimulate further academic progress in stochastic optimization, constrained optimization, probabilistic inference 
with hashing and randomization, and non-convex optimizations with insights from real-world applications. 
Future work includes tightening the constant bound and accelerating the convergence rate. We will keep active to investigate if our approach can motivate new algorithms for non-convex stochastic optimization problems.

%% file: tex/appendix.tex
\section*{Appendix}

\section{XOR-Sampling for the Weighted Case}\label{app:weighted_case}
The text here provides a synopsis for the approach in \cite{ermon2013embed}. We still encourage the readers to read the original text for a better explanation. 
Let $w(\theta), p(\theta)$ and $Z$ as defined before, the high-level idea of XOR-Sampling is to first dicretize $w(\theta)$ to $w'(\theta)$ as in Definition \ref{def:discretize}, followed by embedding the weighted $w'(\theta)$  to  the unweighted space $\Delta_w$. Finally, XOR-sampling uses counting based on hashing and randomization to  sample uniformly from $\Delta_w$. 

\begin{Def}\label{def:discretize}
Assume $w(\theta)$ has both upper and lower bound, namely, $M=\max_{\theta} w(\theta)$ and $m=\min_{\theta} w(\theta)$. Let $b\geq 1, \epsilon>0, r=2^b/(2^b-1)$ and $l=\lceil \log_r(2^n/\epsilon)\rceil$. Partition the configurations into the following weight based disjoint buckets: $\mathcal{B}_i=\{\theta|w(\theta)\in(\frac{M}{r^{i+1}}, \frac{M}{r^{i}}]\}, i=0,\ldots, l-1$ and $\mathcal{B}_l=\{\theta|w(\theta)\in(0,\frac{M}{r^l}]\}$. The discretized weight function $w': \{0,1\}^n\rightarrow\mathbb{R}^+$ is defined as follows: $w'(\theta)=\frac{M}{r^{i+1}}$ if $\theta\in\mathcal{B}_i, i=0,\ldots, l-1$ and $w'(\theta)=0$ if $\theta\in\mathcal{B}_l$. This leads to the corresponding
discretized probability distribution $p'(\theta)=w'(\theta)/Z'$ where $Z'$ is the normalization constant of $w'(\theta)$.
\end{Def}

For the weighted case, the goal of XOR-sampling is to guarantee that the probability of sampling one $\theta$ is proportional to the unnormalized density (up to a multiplicative constant). By Definition \ref{def:discretize}, we obtain a distribution $p'(x)$ which satisfying $\frac{1}{\rho}p(x)\leq p'(x)\leq \rho p(x)$ where $\rho=\frac{r^2}{1-\epsilon}$. Then, XOR-sampling implements a horizontal slice technique to transform a weighted problem into an unweighted one. For the easiness of illustration, we denote $M'=\max_{\theta}w'(\theta)$ and $m'$ as the smallest non-zero value of $w'(\theta)$. Then consider the simple case where $b=1$ and $r=2$, where we have $M'=2^{l-1} m'$. Let $\delta=(\delta_0, \ldots, \delta_{l-2})^T\in\{0,1\}^{l-1}$ be a binary vector of length $l-1$, XOR-sampling samples $(\theta, \delta)$ uniformly at random from the following set $\Delta_w$ using the unweighted version of sampling based on hashing and randomization:
\begin{align}
    \Delta_w=\{(\theta,\delta): w'(\theta)\leq 2^{i+1}m'\Rightarrow\delta_i=0\}.
\end{align}
If we sample $(\theta,\delta)$ uniformly at random from $\Delta_w$ and then only return $\theta$, it can be proved that the probability of sampling $\theta$ from $w'(\theta)$ is proportional to $m'2^{i-1}$ when $w(\theta)$ is sandwiched between $m'2^{i-1}$ and $m'2^i$. Therefore, this technique leads to the constant approximation guarantee of XOR-Sampling.
The precise statement of the guarantee is in Theorem 2. For general case of $b$ and $r$, please refer to \cite{ermon2013embed}.


\noindent\textbf{Setting $\epsilon\eta_{\phi}$ to Zero}~~~~
In Definition \ref{def:discretize} we can make $b$ larger and $\epsilon$ smaller enough, then there will be a possibly large but finite value of $l$ such that $\frac{M}{r^l}$ is smaller than $m$ , which leads $\mathcal{B}_l$ to be empty and $\epsilon\eta_{\phi}$ to be  zero.

\section{Proofs}\label{app:proofs}
\subsection{Proof of Lemma 1} 
We define two functions $g_k^+=\max\{g_k,\mathbf{0}\}$ and $g_k^-=\min\{g_k,\mathbf{0}\}$ where $\mathbf{0}$ is a vector of all 0 which has the same dimension as $g_k$. We have $g_k=g_k^+ + g_k^-$. We define both $\nabla f(x_k)^+$ and $\nabla f(x_k)^-$ in the similar way. Then Lemma 1 gives the new bounds of two terms assuming the constant bound on the gradient, which are essential to the proof of convergence rate. The proof of Lemma 1 is as follows:

\begin{proof} (Lemma 1)
Since we have the constant bound that
\begin{align}
    \frac{1}{c}\nabla f(x_k)^+ & \leq \mathbb{E}[g_k^+]\leq c\nabla f(x_k)^+.\label{eq:pos1}\\
    c\nabla f(x_k)^- & \leq \mathbb{E}[g_k^-]\leq \frac{1}{c}\nabla f(x_k)^-.\label{eq:neg1}
\end{align}
and because of $g_k^+\geq \mathbf{0}$ and $g_k^-\leq \mathbf{0}$ we can obtain
\begin{align*}
     \frac{1}{c}||\mathbb{E}[g_k^+]||_2^2&=\frac{1}{c}\langle \mathbb{E}[g_k^+], \mathbb{E}[g_k^+]\rangle \leq \langle\nabla f(x_k)^+, \mathbb{E}[g_k^+]\rangle \\
     &\leq c\langle \mathbb{E}[g_k^+], \mathbb{E}[g_k^+]\rangle = c||\mathbb{E}[g_k^+]||_2^2.\\
     \frac{1}{c}||\mathbb{E}[g_k^-]||_2^2&=\frac{1}{c}\langle \mathbb{E}[g_k^-], \mathbb{E}[g_k^-]\rangle \leq \langle\nabla f(x_k)^-, \mathbb{E}[g_k^-]\rangle\\
     &\leq c\langle \mathbb{E}[g_k^-], \mathbb{E}[g_k^-]\rangle = c||\mathbb{E}[g_k^-]||_2^2.
\end{align*}
which exactly means
\begin{align*}
     \frac{1}{c} ||\mathbb{E}[g_k]||_2^2 &\leq \langle\nabla f(x_k), \mathbb{E}[g_k]\rangle  \leq c||\mathbb{E}[g_k]||_2^2.
\end{align*}
To prove the second inequality, we need to take advantage of the convexity of $f$. Denote $[x_k-x^*]^+=\max\{x_k-x^*,\textbf{0}\}$ and $[x_k-x^*]^-=\min\{x_k-x^*,\textbf{0}\}$, we know $x_k-x^*=[x_k-x^*]^+ + [x_k-x^*]^-$. In addition, because $f$ is convex, the index set of non-zero entries of $[x_k-x^*]^+$ and $\nabla f(x_k)^+$ is the same. The index set of non-zero entries of $[x_k-x^*]^-$ and $\nabla f(x_k)^-$ is also the same. 
In addition, because of Equation~\ref{eq:pos1} and \ref{eq:neg1}, the index set of non-zero entries of $\mathbb{E}[g_k^+]$ ($\mathbb{E}[g_k^-]$) is the same with 
$\nabla f(x_k)^+$ ($\nabla f(x_k)^-$). 
Combining these facts with Equations \ref{eq:pos1} and \ref{eq:neg1}, we have
\begin{align*}
    \frac{1}{c}\langle\mathbb{E}[g_k^+],[x_k-x^*]^+\rangle  &\leq  \langle\nabla f(x_k)^+,[x_k-x^*]^+\rangle\\
    &\leq c\langle\mathbb{E}[g_k^+],[x_k-x^*]^+\rangle.\\
    \frac{1}{c}\langle\mathbb{E}[g_k^-],[x_k -x^*]^-\rangle  &\leq  \langle\nabla f(x_k)^-,[x_k-x^*]^-\rangle \\
    &\leq c\langle\mathbb{E}[g_k^-],[x_k-x^*]^-\rangle.
\end{align*}
Combining these two equations, we have 
\begin{align*}
    \frac{1}{c}\langle\mathbb{E}[g_k],x_k-x^*\rangle  &\leq  \langle\nabla f(x_k),x_k-x^*\rangle  \leq c\langle\mathbb{E}[g_k],x_k-x^*\rangle.
\end{align*}
This completes the proof.
\end{proof}

\subsection{Proof of Theorem 4}\label{app:proof_main}
\begin{Th}\label{Th:main}
(Main) Let $b,\epsilon,l,\delta,P,\alpha,\rho,\kappa$ and $\mathcal{B}_l$ be as in section \ref{app:weighted_case} in appendix, function $f(x,\theta): \mathbb{R}^d\times\{0,1\}^n\rightarrow \mathbb{R}$ be a $L$-smooth convex function w.r.t. $x$. Denote $OPT=\min_x \mathbb{E}_{\theta\sim Pr(\theta)}f(x,\theta)$ as the global optimum. 
Let $\sigma^2=\max_x\{Var(\nabla_x f(x,\theta))\}$ and $\varepsilon^2=\max_x\{||\mathbb{E}[\nabla_x f(x, \theta)]||_2^2\}$. 
For any $1\leq\rho\kappa\leq\sqrt{2}$, step size $t\leq \frac{2-\rho^2\kappa^2}{L\rho\kappa}$ and sample size $N\geq1$, $\overline{x_K}$ is the output of XOR-SGD and $\mbox{obj} = \mathbb{E}_{\theta}[f(\overline{x_K}, \theta)]$ is the objective function value at $\overline{x_K}$. We have: 
\begin{align}\label{eq:main}
    \mathbb{E}_{\overline{x_K}}[\mbox{obj}] - OPT
     &\leq\frac{\rho\kappa||x_0-x^*||_2^2}{2tK}+\frac{t(\sigma^2+\varepsilon^2)}{N}.
\end{align}
\end{Th}

\begin{proof}(Theorem 4)
Since we use $N$ samples at each iteration, we have $\overline{g_k}=\frac{1}{N}\sum_{i=1}^Ng_k^i$ and $\mathbb{E}[\overline{g_k}]=\mathbb{E}[g_k^i]$.
In each iteration $k$ we can adjust the parameters in XOR-Sampling to make the tail $\epsilon\eta_{\phi}$ zero, then for each sample $g_k^i$ we can obtain from Theorem 2 that
\begin{align}
    \frac{1}{\rho\kappa} \mathbb{E}_\theta [\nabla f(x_k, \theta)]^+ &\leq \mathbb{E}[g_k^{i+}]\leq \rho\kappa\mathbb{E}_\theta [\nabla f(x_k, \theta)]^+.\label{eq:pos2}\\
    \rho\kappa \mathbb{E}_\theta [\nabla f(x_k, \theta)]^- &\leq \mathbb{E}[g_k^{i-}]\leq  \frac{1}{\rho\kappa}\mathbb{E}_\theta [\nabla f(x_k, \theta)]^-.\label{eq:neg2}
\end{align}
The variance of each sample $g_k^i$ can also be bounded by
\begin{align*}
    &Var(g_k^i) \\
    &= \mathbb{E}_{\theta' \sim p'(\theta')} [||\nabla f(x_k, \theta')||_2^2] - ||\mathbb{E}_{\theta' \sim p'(\theta')} [\nabla f(x_k, \theta')]||_2^2,\\
    &\leq \rho\kappa\mathbb{E}_{\theta \sim p(\theta)} [||\nabla f(x_k, \theta)||_2^2],\\
    &= \rho\kappa (Var(\nabla f(x_k,\theta)) + ||\mathbb{E}_{\theta \sim p(\theta)} [\nabla f(x_k, \theta)]||_2^2),\\
    &\leq \rho\kappa(\sigma^2 +\varepsilon^2).
\end{align*}
Denote $\overline{g_k}^+=\max\{\overline{g_k},\textbf{0}\}$ and $\overline{g_k}^-=\min\{\overline{g_k},\textbf{0}\}$. 
Clearly, $g_k^{i+} \geq 0$ and $g_k^{i-} \leq 0$. Moreover,
for a given dimension, either $g_k^{i+}=0$ for that dimension or $g_k^{i-}=0$. 
Evaluating $\overline{g_k}$ dimension by dimension, we can see that
$\overline{g_k}^+=\frac{1}{N}\sum_{i=1}^N g_k^{i+}$ and $\overline{g_k}^-=\frac{1}{N}\sum_{i=1}^N g_k^{i-}$.
Combined with Equation~\ref{eq:pos2} and \ref{eq:neg2}, we know 
\begin{align*}
    \frac{1}{\rho\kappa} \mathbb{E}_\theta [\nabla f(x_k, \theta)]^+ &\leq \mathbb{E}[\overline{g_k}^+]\leq \rho\kappa\mathbb{E}_\theta [\nabla f(x_k, \theta)]^+.\\
    \rho\kappa \mathbb{E}_\theta [\nabla f(x_k, \theta)]^- &\leq \mathbb{E}[\overline{g_k}^-]\leq  \frac{1}{\rho\kappa}\mathbb{E}_\theta [\nabla f(x_k, \theta)]^-.
\end{align*}

Because $\mathbb{E}[\overline{g_k}]=\mathbb{E}[g_k^i]$, we also have
\begin{align*}
    Var(\overline{g_k})=\frac{1}{N^2}Var(\sum_{i=1}^Ng_k^i)=\frac{Var(g_k^i)}{N}.
\end{align*}
Then the variance of $\overline{g_k}$ can be bounded as
\begin{align*}
    Var(\overline{g_k})&\leq \frac{\rho\kappa(\sigma^2+\varepsilon^2)}{N}.
\end{align*}

Therefore, we can then apply Theorem 3 to get the result in equation 5.
\begin{align*}
     &\mathbb{E}_{\overline{x_K}}[\mathbb{E}_{\theta}[f(\overline{x_K}, \theta)]]  -\mathbb{E}_{\theta}[f(x^*,\theta)]\\
     &\leq \frac{\rho\kappa||x_0-x^*||_2^2}{2tK}+\frac{t\max_k\{Var(\overline{g_k})\}}{\rho\kappa},\\
     &\leq\frac{\rho\kappa||x_0-x^*||_2^2}{2tK}+\frac{t(\sigma^2+\varepsilon^2)}{N}.
\end{align*}
which can also be written as
\begin{align}
    \mathbb{E}_{\overline{x_K}}[\mbox{obj}] - OPT
     &\leq\frac{\rho\kappa||x_0-x^*||_2^2}{2tK}+\frac{t(\sigma^2+\varepsilon^2)}{N}.
\end{align}
This completes the proof.
\end{proof}

\section{Experiments}\label{app_exp}
We evaluate our XOR-SGD algorithm on the inventory management \cite{ziukov2016literature,shapiro2007tutorial} and the network design problems \cite{sheldon2012maximizing,WuXSG17XORSampling,xiaojian2016}.
For each setting of both applications, to produce a sample, Gibbs sampling first takes 100 steps to burn in, and then draws samples every 30 steps. We fix the iteration step of both BP and BPChain as $20$, which is enough for BP to converge.
We allow SGD with Gibbs sampling, BP and BPChain to draw more samples than XOR-SGD for a fair comparison. 
All experiments were conducted using single core architectures on Intel Xeon Gold 6126 2.60GHz machines with 96GB RAM and a wall-time limit of $10$ hours. 
For both applications, we use MRF as probabilistic models for $Pr(\theta)$, which can be seen in the next section.
For a fair comparison, once a solution $x$ is generated by 
either algorithm, we use an exact weighted counter ACE \cite{barton2016ace} to evaluate $\mathbb{E}_{\theta\sim Pr(\theta)} f(x, \theta)$ exactly. All objective values reported here are from ACE.

\subsection{Settings of Stochastic Inventory Management}\label{app:exp_setting_inventory}
Taking into account of the storage constraint, the original problem is equivalent to the following problem:
\begin{align}\label{eq:inventory_dual}
    \min_{x\geq 0}\max_{\mu\geq 0}\mathbb{E}_{d\sim Pr(d)}[G(x,d)]+\mu(w^{T}x-X).
\end{align}

For inventory management problem, we assume each $d_i$ can take two different values, one corresponding to the high demand one corresponding to the low demand. Then, we introduce a new vector $\theta$ where $\theta_i=1$ means $d_i$ is the high value while $\theta_i=0$ otherwise. 
%
%
In the experiment we range $n$ from 10 to 100 increased by a step size of 10 and draw 10 instances for each setting. Under each setting, we draw every $c_i$ uniformly from $(0,5]$, $h_i$ uniformly from $(0,10]$, sample $s_i$ uniformly drawn from $(0,10]$ and let $b_i=c_i+s_i$. 
The two values of each $d_i$ are also uniformly drawn from $(0,10]$. 
We model $Pr(\theta)$ as a MRF with several cliques. The variables in each clique are highly correlated with each other. 
For a problem with $n$ products, we draw the number of cliques uniformly from $[n,2n]$. The domain size of each clique $\phi_{\alpha}$ is chosen from the range of $[1,6]$ at random. The potential function of a clique  involving $l$ variables is in the form of a table of size $2^l$. 
The $i$-th entry of this table, denoted as $v_i$, is modeled as $v_i=v_{i1}+v_{i2}v_{i3}$, where $v_{i1}$ is uniformly drawn from $(0,1)$, $v_3$ uniformly from $(10,1000)$ and binary variable $v_{i2}$ uniformly randomly drawn from $\{0,1\}$. 
Each storage requirement $w_i$ is drawn from $(0,10]$ uniformly at random. The largest storage limit $X$ is set to be $5n$. We also evaluate our method given different percentages of the largest storage limit, which is shown in Figure 2 (middle).
In the SGD algorithm, $x$ is initialized with the absolute value of a Gaussian random variable from $\mathcal{N}(5,3)$ to ensure it is non-negative.

\subsection{Settings of Stochastic Network Design}\label{app:exp_setting_network}
The task in equation 8 is equivalent to solving the following problem:
\begin{align}\label{eq:network_dual}
    \min_{\Delta g\geq 0}\max_{\mu\geq 0}~\mathbb{E}_{\theta\sim Pr(\theta)}[\overline{C}(g+\Delta g, \theta)] 
    +\mu(\sum_{e\in E} c_e \Delta g_e- B).
\end{align}
Because of the convexity of $\overline{C}(g+\Delta g, \theta)$ and strong duality, both problems have the same optimal solution.

We test our algorithm on a real-world problem, the so-called Flood Preparation problem for the emergency medical services (EMS) on road networks \cite{xiaojian2016}. The problem setup, including the graph structure and the definition of $Pr(\theta)$, are the same as that in \cite{xiaojian2016}. 
The original network is unweighted, hence we set the initial conductance value for each edge as 1. 
$c_e$ is initialized uniformly from the range $(0,10)$.
The largest budget size $B$ is $1000$. We evaluate our method varying the percentage allowed of the largest budget size, which is shown in Figure 3 (middle).
In the experiment, each entry of $\Delta g$ is initialized with the absolute value of a Gaussian random variable from $\mathcal{N}(0,1)$. Total number of SGD iterations is 2000, while not all algorithms can complete all 2000 iterations within the time limit of 10 hours. 
The experimental results reported in the plots are based on the best solutions found by each algorithm within the time limit. 
Learning rate $t$ is $1$ at first and divided by $10$ after 20 iterations, further by $10$ after 100 iterations. Parameters in XOR-Sampling are set to be the same as in the inventory management problem. 


%% file: uai2022.bbl
\begin{thebibliography}{45}
\providecommand{\natexlab}[1]{#1}
\providecommand{\url}[1]{\texttt{#1}}
\expandafter\ifx\csname urlstyle\endcsname\relax
  \providecommand{\doi}[1]{doi: #1}\else
  \providecommand{\doi}{doi: \begingroup \urlstyle{rm}\Url}\fi

\bibitem[Agarwal et~al.(2017)Agarwal, Allen-Zhu, Bullins, Hazan, and
  Ma]{agarwal2017finding}
Naman Agarwal, Zeyuan Allen-Zhu, Brian Bullins, Elad Hazan, and Tengyu Ma.
\newblock Finding approximate local minima faster than gradient descent.
\newblock In \emph{Proceedings of the 49th Annual ACM SIGACT Symposium on
  Theory of Computing}, pages 1195--1199, 2017.

\bibitem[Allen-Zhu(2017)]{allen2017katyusha}
Zeyuan Allen-Zhu.
\newblock Katyusha: The first direct acceleration of stochastic gradient
  methods.
\newblock \emph{The Journal of Machine Learning Research}, 18\penalty0
  (1):\penalty0 8194--8244, 2017.

\bibitem[Allen-Zhu(2018)]{allen2018natasha}
Zeyuan Allen-Zhu.
\newblock Natasha 2: Faster non-convex optimization than sgd.
\newblock In \emph{Advances in neural information processing systems}, pages
  2675--2686, 2018.

\bibitem[Barton et~al.(2016)Barton, De~Leonardis, Coucke, and
  Cocco]{barton2016ace}
John~P Barton, Eleonora De~Leonardis, Alice Coucke, and Simona Cocco.
\newblock Ace: adaptive cluster expansion for maximum entropy graphical model
  inference.
\newblock \emph{Bioinformatics}, 32\penalty0 (20):\penalty0 3089--3097, 2016.

\bibitem[Ding and Xue(2021)]{fan2021xorsgd}
Fan Ding and Yexiang Xue.
\newblock Xor-sgd: Provable convex stochastic optimization for decision-making.
\newblock In \emph{Uncertainty in Artificial Intelligence}. UAI, 2021.

\bibitem[Ding et~al.(2021)Ding, Ma, Xu, and Xue]{fan2021xorcd}
Fan Ding, Jianzhu Ma, Jinbo Xu, and Yexiang Xue.
\newblock Xor-cd: Linearly convergent constrained structure generation.
\newblock In \emph{International conference on machine learning}. ICML, 2021.

\bibitem[Domke(2013)]{domke2013learning}
Justin Domke.
\newblock Learning graphical model parameters with approximate marginal
  inference.
\newblock \emph{IEEE transactions on pattern analysis and machine
  intelligence}, 35\penalty0 (10):\penalty0 2454--2467, 2013.

\bibitem[Du and Hu(2019)]{du2019linear}
Simon~S Du and Wei Hu.
\newblock Linear convergence of the primal-dual gradient method for
  convex-concave saddle point problems without strong convexity.
\newblock In \emph{The 22nd International Conference on Artificial Intelligence
  and Statistics}, pages 196--205. PMLR, 2019.

\bibitem[Dubey et~al.(2016)Dubey, Reddi, Williamson, Poczos, Smola, and
  Xing]{dubey2016variance}
Kumar~Avinava Dubey, Sashank~J Reddi, Sinead~A Williamson, Barnabas Poczos,
  Alexander~J Smola, and Eric~P Xing.
\newblock Variance reduction in stochastic gradient langevin dynamics.
\newblock In \emph{NIPS}, pages 1154--1162, 2016.

\bibitem[Duchi et~al.(2011)Duchi, Hazan, and Singer]{duchi2011adaptive}
John Duchi, Elad Hazan, and Yoram Singer.
\newblock Adaptive subgradient methods for online learning and stochastic
  optimization.
\newblock \emph{Journal of machine learning research}, 12\penalty0
  (Jul):\penalty0 2121--2159, 2011.

\bibitem[Duchi et~al.(2018)Duchi, Ruan, and Yun]{duchi2018minimax}
John Duchi, Feng Ruan, and Chulhee Yun.
\newblock Minimax bounds on stochastic batched convex optimization.
\newblock In \emph{Conference On Learning Theory}, pages 3065--3162, 2018.

\bibitem[Ermon et~al.(2013{\natexlab{a}})Ermon, Gomes, Sabharwal, and
  Selman]{Ermon13Wish}
Stefano Ermon, Carla~P. Gomes, Ashish Sabharwal, and Bart Selman.
\newblock Taming the curse of dimensionality: Discrete integration by hashing
  and optimization.
\newblock In \emph{Proceedings of the 30th {ICML}}, 2013{\natexlab{a}}.

\bibitem[Ermon et~al.(2013{\natexlab{b}})Ermon, Gomes, Sabharwal, and
  Selman]{ermon2013embed}
Stefano Ermon, Carla~P. Gomes, Ashish Sabharwal, and Bart Selman.
\newblock Embed and project: Discrete sampling with universal hashing.
\newblock In \emph{Advances in Neural Information Processing Systems (NIPS)},
  2013{\natexlab{b}}.

\bibitem[Fan and Xue(2020)]{fan2020BPChainCD}
Ding Fan and Yexiang Xue.
\newblock Contrastive divergence learning with chained belief propagation.
\newblock In \emph{International Conference on Probabilistic Graphical Models},
  2020.

\bibitem[Ge et~al.(2015)Ge, Huang, Jin, and Yuan]{ge2015escaping}
Rong Ge, Furong Huang, Chi Jin, and Yang Yuan.
\newblock Escaping from saddle points—online stochastic gradient for tensor
  decomposition.
\newblock In \emph{Conference on Learning Theory}, pages 797--842, 2015.

\bibitem[Ghosh et~al.(2008)Ghosh, Boyd, and Saberi]{ghosh2008minimizing}
Arpita Ghosh, Stephen Boyd, and Amin Saberi.
\newblock Minimizing effective resistance of a graph.
\newblock \emph{SIAM review}, 50\penalty0 (1):\penalty0 37--66, 2008.

\bibitem[Gomes et~al.(2019)Gomes, Dietterich, Barrett, Conrad, Dilkina, Ermon,
  Fang, Farnsworth, Fern, Fern, et~al.]{gomes2019computational}
Carla Gomes, Thomas Dietterich, Christopher Barrett, Jon Conrad, Bistra
  Dilkina, Stefano Ermon, Fei Fang, Andrew Farnsworth, Alan Fern, Xiaoli Fern,
  et~al.
\newblock Computational sustainability: Computing for a better world and a
  sustainable future.
\newblock \emph{Communications of the ACM}, 62\penalty0 (9):\penalty0 56--65,
  2019.

\bibitem[Hamedani and Aybat(2018)]{hamedani2018primal}
Erfan~Yazdandoost Hamedani and Necdet~Serhat Aybat.
\newblock A primal-dual algorithm with line search for general convex-concave
  saddle point problems.
\newblock \emph{arXiv preprint arXiv:1803.01401}, 2018.

\bibitem[Hinton et~al.(2012)Hinton, Srivastava, and Swersky]{hinton2012neural}
Geoffrey Hinton, Nitish Srivastava, and Kevin Swersky.
\newblock Neural networks for machine learning lecture 6a overview of
  mini-batch gradient descent.
\newblock 2012.

\bibitem[Inman et~al.(2013)Inman, Brock, Inman, Sartorius, Aber, Giddings,
  Cain, Orme, Fredrick, Oakleaf, Alt, Odell, and
  Chapron]{Inman2013DevelopingPF}
Robert~M. Inman, Brent~L. Brock, Kristine~H. Inman, Shawn~S. Sartorius,
  Bryan~C. Aber, Brian Giddings, Steven~L. Cain, Mark~L. Orme, Jay~A. Fredrick,
  Bob~J. Oakleaf, Kurt Alt, Eric~A. Odell, and Guillaume Chapron.
\newblock Developing priorities for metapopulation conservation at the
  landscape scale: Wolverines in the western united states.
\newblock 2013.

\bibitem[Jin et~al.(2017)Jin, Netrapalli, and Jordan]{jin2017accelerated}
Chi Jin, Praneeth Netrapalli, and Michael~I Jordan.
\newblock Accelerated gradient descent escapes saddle points faster than
  gradient descent.
\newblock \emph{arXiv preprint arXiv:1711.10456}, 2017.

\bibitem[Kingma and Ba(2014)]{kingma2014adam}
Diederik~P Kingma and Jimmy Ba.
\newblock Adam: A method for stochastic optimization.
\newblock \emph{arXiv preprint arXiv:1412.6980}, 2014.

\bibitem[Lee et~al.(2015)Lee, Lin, Ma, and Yang]{lee2015distributed}
Jason~D Lee, Qihang Lin, Tengyu Ma, and Tianbao Yang.
\newblock Distributed stochastic variance reduced gradient methods and a lower
  bound for communication complexity.
\newblock \emph{arXiv preprint arXiv:1507.07595}, 2015.

\bibitem[Liu and Ihler(2013)]{LiuI13VariationalMMAP}
Qiang Liu and Alexander~T. Ihler.
\newblock Variational algorithms for marginal {MAP}.
\newblock \emph{Journal of Machine Learning Research}, 14, 2013.

\bibitem[Liu et~al.(2017)Liu, Shang, and Cheng]{liu2017accelerated}
Yuanyuan Liu, Fanhua Shang, and James Cheng.
\newblock Accelerated variance reduced stochastic admm.
\newblock In \emph{Proceedings of the AAAI Conference on Artificial
  Intelligence}, volume~31, 2017.

\bibitem[Marinescu et~al.(2014)Marinescu, Dechter, and Ihler]{Radu14AndOrMMAP}
Radu Marinescu, Rina Dechter, and Alexander~T. Ihler.
\newblock {AND/OR} search for marginal {MAP}.
\newblock In \emph{Proceedings of the Thirtieth Conference on Uncertainty in
  Artificial Intelligence, {UAI}}, 2014.

\bibitem[Marinescu et~al.(2015)Marinescu, Dechter, and
  Ihler]{Marinescu2015AOBB}
Radu Marinescu, Rina Dechter, and Alexander Ihler.
\newblock Pushing forward marginal map with best-first search.
\newblock In \emph{Proceedings of the 24th International Conference on
  Artificial Intelligence (IJCAI)}, 2015.

\bibitem[Mau{\'{a}} and de~Campos(2012)]{Maua2012AnytimeMAP}
Denis~Deratani Mau{\'{a}} and Cassio~Polpo de~Campos.
\newblock Anytime marginal {MAP} inference.
\newblock In \emph{Proceedings of the 29th ICML}, 2012.

\bibitem[McClure et~al.(2016)McClure, Hansen, and
  Inman]{McClure2016ConnectingMT}
Meredith~L McClure, Andrew~J. Hansen, and Robert~M. Inman.
\newblock Connecting models to movements: testing connectivity model
  predictions against empirical migration and dispersal data.
\newblock \emph{Landscape Ecology}, 31:\penalty0 1419--1432, 2016.

\bibitem[Mokhtari et~al.(2020)Mokhtari, Ozdaglar, and
  Pattathil]{mokhtari2020convergence}
Aryan Mokhtari, Asuman~E Ozdaglar, and Sarath Pattathil.
\newblock Convergence rate of o(1/k) for optimistic gradient and extragradient
  methods in smooth convex-concave saddle point problems.
\newblock \emph{SIAM Journal on Optimization}, 30\penalty0 (4):\penalty0
  3230--3251, 2020.

\bibitem[Murphy et~al.(2013)Murphy, Weiss, and Jordan]{murphy2013loopy}
Kevin Murphy, Yair Weiss, and Michael~I Jordan.
\newblock Loopy belief propagation for approximate inference: An empirical
  study.
\newblock \emph{arXiv preprint arXiv:1301.6725}, 2013.

\bibitem[Ouyang et~al.(2013)Ouyang, He, Tran, and Gray]{ouyang2013stochastic}
Hua Ouyang, Niao He, Long Tran, and Alexander Gray.
\newblock Stochastic alternating direction method of multipliers.
\newblock In \emph{International Conference on Machine Learning}, pages 80--88.
  PMLR, 2013.

\bibitem[Ruder(2016)]{ruder2016overview}
Sebastian Ruder.
\newblock An overview of gradient descent optimization algorithms.
\newblock \emph{arXiv preprint arXiv:1609.04747}, 2016.

\bibitem[Shapiro and Philpott(2007)]{shapiro2007tutorial}
Alexander Shapiro and Andy Philpott.
\newblock A tutorial on stochastic programming.
\newblock 2007.

\bibitem[Sheldon et~al.(2012)Sheldon, Dilkina, Elmachtoub, Finseth, Sabharwal,
  Conrad, Gomes, Shmoys, Allen, Amundsen, et~al.]{sheldon2012maximizing}
Daniel Sheldon, Bistra Dilkina, Adam~N Elmachtoub, Ryan Finseth, Ashish
  Sabharwal, Jon Conrad, Carla~P Gomes, David Shmoys, William Allen, Ole
  Amundsen, et~al.
\newblock Maximizing the spread of cascades using network design.
\newblock \emph{arXiv preprint arXiv:1203.3514}, 2012.

\bibitem[Sodomka et~al.(2007)Sodomka, Collins, and Gini]{sodomka2007efficient}
Eric Sodomka, John Collins, and Maria Gini.
\newblock Efficient statistical methods for evaluating trading agent
  performance.
\newblock 2007.

\bibitem[Wang et~al.(2013)Wang, Chen, Smola, and Xing]{wang2013variance}
Chong Wang, Xi~Chen, Alexander~J Smola, and Eric~P Xing.
\newblock Variance reduction for stochastic gradient optimization.
\newblock In \emph{Advances in Neural Information Processing Systems}, pages
  181--189, 2013.

\bibitem[Wang and Li(2020)]{wang2020improved}
Yuanhao Wang and Jian Li.
\newblock Improved algorithms for convex-concave minimax optimization.
\newblock \emph{arXiv preprint arXiv:2006.06359}, 2020.

\bibitem[Wu et~al.(2016)Wu, Sheldon, and Zilberstein]{xiaojian2016}
Xiaojian Wu, Daniel~R Sheldon, and Shlomo Zilberstein.
\newblock Optimizing resilience in large scale networks.
\newblock In \emph{Proceedings of the 30th Conference of AAAI}, 2016.

\bibitem[Wu et~al.(2017)Wu, Xue, Selman, and Gomes]{WuXSG17XORSampling}
Xiaojian Wu, Yexiang Xue, Bart Selman, and Carla~P. Gomes.
\newblock Xor-sampling for network design with correlated stochastic events.
\newblock In \emph{Proceedings of the 26th IJCAI}, pages 4640--4647, 2017.

\bibitem[Xie et~al.(2020)Xie, Luo, Lian, and Zhang]{xie2020lower}
Guangzeng Xie, Luo Luo, Yijiang Lian, and Zhihua Zhang.
\newblock Lower complexity bounds for finite-sum convex-concave minimax
  optimization problems.
\newblock In \emph{International Conference on Machine Learning}, pages
  10504--10513. PMLR, 2020.

\bibitem[Xue et~al.(2016)Xue, Li, Ermon, Gomes, and Selman]{Xue2016MarginalMAP}
Yexiang Xue, Zhiyuan Li, Stefano Ermon, Carla~P. Gomes, and Bart Selman.
\newblock Solving marginal map problems with np oracles and parity constraints.
\newblock In \emph{Proceedings of the 29th Annual Conference on NIPS}, 2016.

\bibitem[Yedidia et~al.(2001)Yedidia, Freeman, and
  Weiss]{yedidia2001generalized}
Jonathan~S Yedidia, William~T Freeman, and Yair Weiss.
\newblock Generalized belief propagation.
\newblock In \emph{Advances in neural information processing systems}, pages
  689--695, 2001.

\bibitem[Zheng and Kwok(2016)]{zheng2016fast}
Shuai Zheng and James~T Kwok.
\newblock Fast-and-light stochastic admm.
\newblock In \emph{IJCAI}, pages 2407--2613, 2016.

\bibitem[Ziukov(2016)]{ziukov2016literature}
Serhii Ziukov.
\newblock A literature review on models of inventory management under
  uncertainty.
\newblock 2016.

\end{thebibliography}
